\documentclass[reqno,twoside]{amsart}
\usepackage{amscd}
\usepackage{graphicx}
\usepackage{amsfonts}
\usepackage{amsmath}
\usepackage{amssymb}
\usepackage{amsthm}
\usepackage{amsaddr}
\usepackage{fancyhdr}  
\usepackage{latexsym}
\usepackage[colorlinks=true, pdfstartview=FitV, linkcolor=blue, citecolor=blue, urlcolor=blue]{hyperref}
\usepackage{enumitem}      
\usepackage{mathtools}            
\usepackage{indentfirst} 
\usepackage{color}
\usepackage{stackengine}
\usepackage{framed}
\usepackage{caption}          
\usepackage[normalem]{ulem}
\usepackage{upgreek}
\usepackage{dutchcal}
\usepackage{mathrsfs}
\usepackage{tikz}
\usetikzlibrary{decorations.markings, decorations.pathmorphing, arrows.meta}
\tikzset{
    >=Stealth,
    branch cut/.style={decorate, decoration={snake, amplitude=1pt, segment length=3pt}},
    mid arrow/.style={postaction={decorate, decoration={markings, mark=at position #1 with {\arrow{>}}}}},
    mid arrow/.default=0.5
}
\usepackage{multicol}
\newenvironment{Figure}
  {\par\medskip\noindent\minipage{\linewidth}}
  {\endminipage\par\medskip}
\newcommand{\eee}[1]{\begin{equation}#1\end{equation}}

\newcommand{\aaa}[1]{\begin{alignat}{2}#1\end{alignat}}
\newcommand{\ddd}[1]{\begin{equation}\begin{aligned}#1\end{aligned}\end{equation}}
\newcommand{\R}{\mathbb{R}}
\newcommand{\N}{\mathbb{N}}
\newcommand{\C}{\mathbb{C}}
\newcommand{\nn}{\nonumber}
\newcommand{\p}{\partial}

\newcommand{\no}[1]{\left\| #1 \right\|}
\newcommand{\floor}[1]{\lfloor #1 \rfloor}
\newcommand{\ceil}[1]{\lceil #1 \rceil}
\newcommand{\what}{\widehat}
\newcommand{\wtil}{\widetilde}
\renewcommand{\leq}{\leqslant}
\renewcommand{\geq}{\geqslant}
\usepackage{accents}
\renewcommand*{\dot}[1]{%
\accentset{\mbox{\small\bfseries .}}{#1}}
\usepackage{scalerel,stackengine}
\stackMath
\newcommand\wwhat[1]{%
\savestack{\tmpbox}{\stretchto{%
  \scaleto{%
    \scalerel*[\widthof{\ensuremath{#1}}]{\kern-.6pt\bigwedge\kern-.6pt}%
    {\rule[-\textheight/2]{1ex}{\textheight}}
  }{\textheight}%
}{0.5ex}}%
\stackon[1pt]{#1}{\tmpbox}%
}
\usepackage{thmtools}
\declaretheoremstyle[headfont=\normalfont\bfseries, bodyfont=\itshape,spaceabove=7pt, spacebelow=7pt]{theorem} 
\theoremstyle{theorem} 
\newtheorem{theorem}{Theorem}[section] 

\newtheorem{lemma}{Lemma}[section]

\theoremstyle{definition}


\allowdisplaybreaks
\numberwithin{equation}{section}
\numberwithin{figure}{section}
\usepackage{geometry}
\usepackage[breakable, theorems, skins]{tcolorbox}
\tcbset{enhanced}

\usepackage{tikz}
\usetikzlibrary{arrows.meta}
\usetikzlibrary{decorations.markings}
\tikzset{->-/.style={decoration={
  markings,
  mark=at position #1 with {\arrow{>}}},postaction={decorate}}}
  \tikzset{middlearrow/.style={
        decoration={markings,
            mark= at position 0.55 with {\arrow{#1}} ,
        },
        postaction={decorate}
    }
}
\makeatletter
  \renewcommand\subsection{\@startsection{subsection}{2}%
  \z@{-1\linespacing\@plus-0.7\linespacing}{0.7\linespacing}%
  {\bfseries}}
\makeatother
\let\OLDthebibliography\thebibliography
\renewcommand\thebibliography[1]{
  \OLDthebibliography{#1}
  \setlength{\parskip}{0pt}
  \setlength{\itemsep}{0pt plus 0.3ex}
}
\AtBeginDocument{
   \def\MR#1{}
}

\begin{document}

\title{The ``good'' Boussinesq equation on the half-line \\ with Robin boundary conditions}

\author{Shivani Agarwal and Dionyssios Mantzavinos}
	
\address{
\normalfont Department of Mathematics, University of Kansas, Lawrence, KS 66045, USA
} 
\email{\!shivaniagarwal@ku.edu, mantzavinos@ku.edu \textnormal{(corresponding author)}}

\thanks{\textit{Acknowledgements.} The authors gratefully acknowledge support from the U.S. National Science Foundation (grants NSF-DMS 2206270 and NSF-DMS 2509146). DM is also thankful to the Simons Foundation (award SFI-MPS-TSM-00013970).
}
\date{May 13, 2026}

\begin{abstract}
We prove the local Hadamard well-posedness of the ``good'' Boussinesq equation formulated on the half-line with nonzero Robin boundary conditions. These boundary data involve the Dirichlet and Neumann boundary values as well as the second spatial derivative of the solution evaluated at the boundary. 
The nonlinear analysis crucially relies on the linear estimates established through the explicit solution formula obtained for the forced linear counterpart of the problem via Fokas's unified transform.
The two pieces of initial data and the two pieces of boundary data belong in appropriate Sobolev spaces. 
The corresponding solution is established in the natural Hadamard solution space of continuous/continuously differentiable functions from a suitable time interval to the Sobolev spaces associated with the two initial data. Furthermore, in line with the well-posedness theory of the Cauchy problem, in the case of low regularity (namely, below the spatial continuity threshold) the solution space is refined by also including an appropriate spatiotemporal Lebesgue space.
\end{abstract}

\maketitle
\markboth
{Shivani Agarwal and Dionyssios Mantzavinos}
{The ``good'' Boussinesq equation on the half-line with Robin boundary conditions}


\section{Introduction}

We establish local Hadamard well-posedness (i.e. the existence of a unique solution that depends continuously on the data) for the ``good'' Boussinesq equation formulated on the half-line with nonzero Robin boundary data:
\ddd{\label{gbous}
&u_{tt} -  u_{xx} + u_{xxxx} + (u^2)_{xx} = 0, \quad u = u(x, t),\quad x \in (0, \infty), \ t\in (0, T), 
\\
&u(x, 0) = u_0(x), 
\quad
u_t(x, 0) = u_1(x), 
\\
& u_x(0,t)+\gamma  u(0,t) = \alpha(t), \quad u_{xx}(0,t)+ \delta u_{x}(0,t) = \beta(t).
}
In the above initial-boundary value problem,  $\gamma, \delta \in \R$ are constants and the lifespan $T>0$ satisfies an appropriate condition that involves the initial and boundary data via suitable Sobolev norms (see Theorems \ref{lwp-t} and~\ref{lwp-t-low} below). The two nonzero Robin boundary conditions involve the Dirichlet and Neumann values as well as the second spatial derivative of the solution evaluated at the boundary $x=0$. None of these three boundary values is individually specified; rather, two linear combinations of them are prescribed via the Robin data $\alpha, \beta$.

The ``good'' Boussinesq equation can be used for modeling nonlinear vibrations along a string \cite{z1973} as well as electromagnetic waves in nonlinear dielectric materials \cite{t1993}. Furthermore,  it is a completely integrable system~\cite{z1973} so, in particular, it possesses a Lax pair and can be studied via the inverse scattering transform~\cite{dtt1982}. The roots of the equation can be traced back to Boussinesq's seminal work \cite{b1872} on the theoretical foundation of Scott Russell's observation of the ``great wave of translation'', nowadays known as the soliton \cite{sr1838}. In \cite{b1872}, the equation is derived in the context of water waves in the form of the ``bad'' Boussinesq equation $u_{tt} - u_{xx} - u_{xxxx} + (u^2)_{xx} = 0$. However, as the dispersion relation of this latter equation becomes imaginary for wavenumbers of size greater than one, the associated initial value problem is ill-posed. A change of sign in the third term results in a real dispersion relation for all real wavenumbers and hence in a well-posed model in the form of the ``good'' Boussinesq equation given in \eqref{gbous}.

Prior to stating our results, we recall the definition of the Sobolev space $H^s(\Omega)$ as the restriction on any open set $\Omega\subset \R$ of the Sobolev space $H^s(\R)$ defined by
\eee{
H^s(\R) := \left\{f \in \mathcal S'(\R): \left(1+k^2\right)^{\frac s2} \mathcal F\{f\}(k) \in L^2(\R)\right\},
\quad
s\in\R,
}
with norm $\no{f}_{H^s(\R)} := \big\| \left(1+k^2\right)^{\frac s2} \mathcal F\{f\}(k) \big\|_{L^2(\R)}$, 
where $\mathcal S'(\R)$ denotes the space of tempered distributions on $\R$ and $\mathcal F\{f\}$ is the Fourier transform of $f$ over $\R$ defined by
\eee{
\mathcal F\{f\}(k) := \int_{\R} e^{-ikx} f(x) dx, \quad k \in \R.
}
More precisely, denoting by $\mathcal D'(\Omega)$ the space of distributions on $\Omega$, we define
\eee{
H^s(\Omega) :=  \left\{f \in \mathcal D'(\Omega): f = F|_\Omega \text{ with } F \in H^s(\R)\right\}, \quad s\in \R,
}
with corresponding norm
$\no{f}_{H^s(\Omega)} := \inf\big\{\no{F}_{H^s(\R)}: F|_\Omega = f \big\}$.
With this definition at hand, we introduce the Banach solution space
\ddd{\label{xst}
&X_T^s := C_t((0,T);H^{s}_x(0, \infty)) \cap  C_t^1((0,T);H^{s-2}_x(0, \infty)),
\\
&\no{u}_{X_T^s} = \no{u}_{L_t^\infty((0, T); H_x^s(0, \infty))}  + \no{\p_t u}_{L_t^\infty((0, T); H_x^{s-2}(0, \infty))}.
}
Then, our local well-posedness result for the nonlinear problem \eqref{gbous} in the case of high regularity (i.e. above the spatial continuity threshold of $s>\frac 12$) can be stated as follows:
\begin{theorem}[Local well-posedness  -- high regularity]
\label{lwp-t}
For $\frac 12 < s < \frac 92$ with $s\notin \N + \frac 12$, consider the  nonlinear problem \eqref{gbous} for the ``good'' Boussinesq equation on the half-line with initial data $(u_0, u_1) \in H^s(0, \infty) \times H^{s-2}(0, \infty)$ and Robin boundary data $(\alpha, \beta) \in H^{\frac{2s-1}{4}}(0, T) \times H^{\frac{2s-3}{4}}(0, T)$ satisfying the compatibility conditions
\eee{\label{comp}
u_0'(0) + \gamma u_0(0) = \alpha(0), \ s > \tfrac 32,
\quad
u_0''(0) + \delta u_0'(0) = \beta(0), \  s > \tfrac 52, 
\quad
u_1'(0) + \gamma u_1(0) = \alpha'(0), \  s > \tfrac 72.
}
Then, for lifespan $T>0$ such that
\eee{\label{contr-cond}
c(s,T)^2 \Big( \no{u_0}_{H^{s}(0, \infty)}
+
\no{u_1}_{H^{s-2}(0, \infty)}
+
\no{\alpha}_{H^{\frac{2s-1}{4}}(0, T)} 
+
\no{\beta}_{H^{\frac{2s-3}{4}}(0, T)}  \Big) \sqrt T < 1
}
with $c(s, T)>0$ a certain constant that remains bounded as $T\to 0^+$, the problem \eqref{gbous} has a unique solution $u \in X_T^s$, which admits the size estimate
\eee{
\no{u}_{X_T^s} 
\leq 
c(s,T) \Big( \no{u_0}_{H^{s}(0, \infty)}
+
\no{u_1}_{H^{s-2}(0, \infty)}
+
\no{\alpha}_{H^{\frac{2s-1}{4}}(0, T)} 
+
\no{\beta}_{H^{\frac{2s-3}{4}}(0, T)}  \Big).
}
Furthermore, the data-to-solution map $\{u_0, u_1, \alpha, \beta\} \mapsto u$ is locally Lipschitz continuous.
\end{theorem}

The compatibility conditions \eqref{comp} are required due to the Sobolev regularity of the initial and boundary data specified in Theorem \ref{lwp-t}.
In this connection, we note that for $s<\frac 92$ no additional compatibility conditions arise through the ``good'' Boussinesq equation itself.
Moreover, we note that the Sobolev exponents of the four pieces of initial and boundary data involved in Theorem \ref{lwp-t} are related to each other in an \textit{optimal way}, which is determined by estimating the solution to the forced linear counterpart of the nonlinear problem \eqref{gbous}, namely
\ddd{\label{lgbous}
&u_{tt} - u_{xx} + u_{xxxx} = -\p_x^2 w , \quad  x \in (0, \infty), \ t\in (0, T), 
\\
&u(x,0) = u_0(x),
\quad
u_t(x,0) = u_1(x), 
\\
&u_x(0,t)+\gamma u(0,t)=\alpha(t), \quad u_{xx}(0,t)+\delta u_{x}(0,t)=\beta(t), 
}
where $w = w(x, t)$ is a given forcing function.
Theorem \ref{lwp-t} will be proved via a contraction mapping argument that crucially relies on the following result for the forced linear  problem \eqref{lgbous}.
\begin{theorem}[Linear estimate -- high regularity]\label{linear-t}
For $\frac12 < s < \frac92$ with $s\notin \N + \frac12$ and any $T>0$, the solution to the problem \eqref{lgbous} for the forced linear ``good'' Boussinesq equation on the half-line with initial and boundary data as in Theorem \ref{lwp-t}, forcing $w \in L_t^2((0, T); H_x^s(0, \infty))$, and the compatibility conditions \eqref{comp} in place, admits the estimate
\aaa{\label{lin-est}
\no{S\big[u_0, u_1, \alpha, \beta; -\p_x^2 w\big]}_{X_T^s}
\leq
c(s, T) 
\Big(
&\no{u_0}_{H^{s}(0, \infty)}
+
\no{u_1}_{H^{s-2}(0, \infty)}
+
\no{\alpha}_{H^{\frac{2s-1}{4}}(0, T)} 
+
\no{\beta}_{H^{\frac{2s-3}{4}}(0, T)} 
\nn\\
&+
\no{w}_{L_t^2((0,T);H^{s}_x(0, \infty))}\Big)
}
where $c(s, T)>0$ is a constant that remains bounded as $T\to 0^+$. 
\end{theorem}

In the nonlinear analysis leading to the proof of Theorem \ref{lwp-t}, the linear estimate \eqref{lin-est} of Theorem \ref{linear-t} will be employed with  $w = u^2$ and will be combined with the Sobolev algebra property. 
However, in the setting of low regularity corresponding to $s<\frac 12$, the algebra property is no longer present. In that case, the nonlinearity $w=u^2$ will be handled by replacing the solution space $X_T^s$ defined by \eqref{xst} with the refined space
\ddd{\label{yst}
&Y_T^s := C_t((0,T);H_x^s(0, \infty))  \cap C_t^1((0, T); H_x^{s-2}(0, \infty)) \cap L_t^4((0,T);L_x^\infty(0, \infty)),
\\
&\no{u}_{Y_T^s} := \no{u}_{L_t^\infty((0, T); H_x^s(0, \infty))}  + \no{\p_t u}_{L_t^\infty((0, T); H_x^{s-2}(0, \infty))} + \no{u}_{L_t^4((0, T); L_x^\infty(0, \infty))},
}
and our local well-posedness result for the nonlinear problem \eqref{gbous} can be stated as follows:
\begin{theorem}[Local well-posedness -- low regularity]
\label{lwp-t-low}
For $0 \leq s < \frac12$, consider the  nonlinear problem \eqref{gbous} for the ``good'' Boussinesq equation on the half-line with initial data $u_0 \in H^s(0, \infty)$, $u_1 = v_1' \in H^{s-2}(0, \infty)$ where $v_1 \in H^{s-1}(0, \infty)$ and Robin boundary data $\alpha  \in H^{\frac{2s-1}{4}}(0, T)$, $\beta = b' \in H^{\frac{2s-3}{4}}(0, T)$ where $b \in H^{\frac{2s+1}{4}}(0, T)$.
Then, for lifespan $T>0$ such that
\eee{
c(s,T)^2 \Big( \no{u_0}_{H^{s}(0, \infty)}
+
\no{v_1}_{H^{s-1}(0, \infty)}
+
\no{\alpha}_{H^{\frac{2s-1}{4}}(0, T)} 
+
\no{b}_{H^{\frac{2s+1}{4}}(0, T)}  \Big) \sqrt T < 1
}
with $c(s, T)>0$ a certain constant that remains bounded as $T\to 0^+$, the problem~\eqref{gbous} has a unique solution $u \in \overline{B(0, \rho)} \subset Y_T^s$, where $B(0, \rho)$ denotes the ball centered at zero and of radius
\eee{
\rho = \rho(s, T) := c(s,T) \Big( \no{u_0}_{H^s(0,\infty)} +\no{v_1}_{H^{s-1}(0,\infty)} + \no{\alpha}_{H^{\frac{2s-1}{4}}(0,T)} + \no{b}_{H^\frac{2s+1}{4}(0,T)}\Big).
}
Furthermore, the data-to-solution map $\{u_0, v_1, \alpha, b\} \mapsto u$ is locally Lipschitz continuous.
\end{theorem}

We note that the assumption $u_1 = v_1'$ for the second initial datum when $s<\frac 12$ is also present in the analysis of the Cauchy (initial value) problem for the ``good'' Boussinesq equation on the infinite line \cite{l1993}. However, our analysis shows that this assumption can actually be avoided in the context of the Cauchy problem. On the other hand, the assumption $\beta=b'$ for the second Robin datum is necessary for employing a certain Sobolev extension theorem when $s<\frac 12$ (see beginning of Section \ref{l-nl-low-s}) and it is the reason why we also assume $u_1 = v_1'$ in that range of~$s$.
As in the case of high regularity, the proof of Theorem \ref{lwp-t-low} is based on a contraction mapping argument that utilizes the following result for the forced linear problem \eqref{lgbous}:
\begin{theorem}[Linear estimate -- low regularity]\label{linear-low-t}
For $0 \leq s < \frac12$ and any $T>0$, the solution to the problem~\eqref{lgbous} for the forced linear ``good'' Boussinesq equation on the half-line with initial and boundary data as in Theorem \ref{lwp-t-low} and forcing $w \in L_t^1((0, T); H_x^s(0, \infty))$ admits the estimate
\aaa{\label{lin-est-low}
\no{S\big[u_0, v_1', \alpha, b'; -\p_x^2 w\big]}_{Y_T^s}
\leq
c(s, T) \Big(
&\no{u_0}_{H^{s}(0, \infty)} + \no{v_1}_{H^{s-1}(0, \infty)}
+
\no{\alpha}_{H^{\frac{2s-1}{4}}(0, T)}
+
\no{b}_{H^{\frac{2s+1}{4}}(0, T)}
\nn\\
&+\no{w}_{L_t^1((0, T); H_x^s(0, \infty))}
 \Big)
 }
where $c(s, T)>0$ is a constant that remains bounded as $T\to 0^+$. 
\end{theorem}

The crucial linear estimates \eqref{lin-est} and \eqref{lin-est-low} are  proved via the explicit solution formula \eqref{T-sol} for the forced linear initial-boundary value problem~\eqref{lgbous}, which is derived via the unified transform (also known as the Fokas method) \cite{f1997,f2008}. This is a powerful technique designed specifically for the solution of linear initial-boundary value problems supplemented with any type of admissible boundary condition. It provides the direct analogue of the Fourier transform (which is used in the analysis of initial value problems) for initial-boundary value problems. Since its introduction nearly thirty years ago, there are numerous related contributions in various directions; indicatively, we mention \cite{f2002,p2004,fp2005,s2012,fps2022} as well as the expository works and conference proceedings volumes~\cite{dtv2014,fp2015,bountis2022,cfk2025} and the references therein. 

In the present work, we prove Theorems \ref{lwp-t}-\ref{linear-low-t} by taking advantage of a concrete application of the unified transform to nonlinear partial differential equations, namely of a novel method introduced by one of the authors during the last decade for proving the Hadamard local well-posedness of nonlinear evolution equations in the setting of nonhomogeneous initial-boundary value problems. This method was originally developed for the nonlinear Schr\"odinger and Korteweg-de Vries equations on the half-line in the case of Dirichlet data~\cite{fhm2017,fhm2016}. Subsequently, it was employed for the rigorous analysis of several other nonlinear problems, e.g. see \cite{hmy2019-rd,oy2019,hm2020,hy2022-jde,amo2024,mo2025,mmo2026}. In particular, it was used for proving the local well-posedness of the ``good'' Boussinesq equation on the half-line in the case of Dirichlet and Neumann boundary data, i.e. when prescribing $u(0, t)$ and $u_x(0, t)$ \cite{hm2015}. 
Here, we deal with the more demanding case of Robin boundary data, which come in the form of linear combinations of the Dirichlet value, the Neumann value, and the second spatial derivative at $x=0$. The presence of this latter term results in a rougher Sobolev space for the second boundary condition. Moreover, the presence of Robin data results in a more complicated solution formula that involves integrands with singularities in the complex spectral plane (note the terms $\Delta_{1, 2}$ in \eqref{T-sol}). These new features make the analysis more involved from a technical standpoint.

Beyond the different type of boundary data, the present work advances \cite{hm2015} in two further directions. First, in the high regularity setting of $s>\frac 12$ (which is the one considered in \cite{hm2015}), it removes the assumption of~\cite{hm2015} that the second initial datum as well as the Neumann datum be prescribed in the form of a spatial and a temporal derivative respectively. Namely, contrary to \cite{hm2015}, the well-posedness result of Theorem~\ref{lwp-t} does not require any additional structure on any of the four pieces of initial and boundary data.
Second, importantly, Theorems~\ref{lwp-t-low} and~\ref{linear-low-t} advance the local well-posedness theory of the ``good'' Boussinesq equation to the Strichartz estimates territory of low regularity $s<\frac 12$, in contrast with \cite{hm2015} which is restricted to the high regularity setting of $s>\frac 12$. 

This latter improvement can be especially appreciated by noting that \textit{Theorem \ref{lwp-t-low} provides the precise analogue of the best result available in classical (i.e. non-Bourgain) spaces for the ``good'' Boussinesq equation in the setting of the Cauchy problem, obtained by Linares  \cite{l1993} via the use of appropriate Strichartz estimates}. In this connection, we mention that the well-posedness of the Cauchy problem for the ``good'' Boussinesq equation on the whole line has been studied extensively. In particular, local well-posedness for initial data in Sobolev spaces of higher regularity was established by Bona and Sachs~\cite{bs1988}, while lower regularity results in classical spaces were subsequently obtained by Tsutsumi and Matahashi \cite{tm1991} and, notably, Linares \cite{l1993}. In later years, Farah~\cite{farah2009} as well as Kishimoto and Tsugawa \cite{kt2010} were able to penetrate into the territory of negative Sobolev exponents via the use of Bourgain spaces. Corresponding results for the periodic Cauchy problem were obtained by Farah and Scialom~\cite{fs2010}, Oh and Stefanov \cite{os2013}, and Kishimoto \cite{k2013}. 

In light of the above, it should be emphasized that the study of nonlinear evolution equations supplemented with nonzero boundary conditions involves several additional challenges, the most fundamental one being the lack of Fourier transform and corresponding harmonic analysis machinery. This fact possibly explains why the corresponding literature is in general much more limited than the one on the Cauchy problem. Specifically, besides the method of \cite{fhm2017,fhm2016} employed here, which circumvents the lack of Fourier transform by  means of Fokas's unified transform,  two other main approaches for the well-posedness of nonlinear initial-boundary value problems are (i) the temporal Laplace transform method of Bona, Sun and Zhang, which was introduced in \cite{bsz2002} for the KdV equation on the half-line and has since been used in several other works, e.g.~\cite{bsz2006,bsz2008,kai2013,ozs2015,bo2016,et2016}, and (ii) the  boundary forcing operator method of Colliander,  Kenig and Holmer, first developed for the generalized KdV equation on the half-line \cite{ck2002} and subsequently for the KdV and NLS equations on the half-line \cite{h2005,h2006,c2017,cc2020}. Other noteworthy contributions are those by Faminskii \cite{fam2004,fam2007,f2020,f2024}. 

Specifically for the ``good'' Boussinesq equation on the half-line, other well-posedness works besides \cite{hm2015} are those by Xue \cite{x2008,x2010} (including results in the low regularity setting of $s<\frac 12$) as well as by Compaan and Tzirakis \cite{ct2017}, all three of them via the temporal Laplace transform method of \cite{bsz2002} and in the case of \textit{individual boundary values} specified as data. The well-posedness results established in the present work for the \textit{Robin data} problem \eqref{gbous} are, to the best of our knowledge, the first in this direction, following corresponding results for the nonlinear Schr\"odinger equation obtained in \cite{hm2021} but, in addition, advancing the analysis into the low regularity territory of $s<\frac 12$.
\\[2mm]
\textit{Structure.} In Section \ref{red-s}, we estimate a reduced initial-boundary value problem with has zero forcing, zero initial data and compactly supported boundary data. This task reveals the function spaces for the Robin boundary data in problem \eqref{gbous}. In Section \ref{cauchy-s}, we proceed to the estimation of the homogeneous and forced Cauchy problems where, importantly, we establish new time estimates needed specifically in the context of the initial-boundary value problem analysis. We then prove the high regularity results of Theorems \ref{lwp-t} and \ref{linear-t} in Section~\ref{l-nl-high-s}. In Section \ref{l-nl-low-s}, we deal with the case of low regularity and obtain new linear estimates in the refined solution space $Y_T^s$, leading to the proof of Theorems \ref{lwp-t-low} and \ref{linear-low-t}. Finally, in Section \ref{ut-s} we derive the unified transform solution formula \eqref{T-sol} for the forced linear problem \eqref{lgbous}, which plays a crucial role in the overall analysis presented in this work.

\section{Estimation of the reduced initial-boundary value problem}
\label{red-s}

In this section, we estimate the solution to the reduced initial-boundary value problem
\ddd{\label{lgbous-r}
&v_{tt} - v_{xx} + v_{xxxx} = 0, \quad  x \in (0, \infty), \ t\in (0, T), 
\\
&v(x,0) = 0, 
\quad
v_t(x,0) = 0, 
\\
&v_x(0,t)+\gamma v(0,t) = \varphi(t), \quad v_{xx}(0,t)+\delta v_{x}(0,t) = \psi(t), 
\\
&\text{supp}(\varphi), \ \text{supp}(\psi) \subset (0, T), 
}
whose solution follows from the general formula \eqref{T-sol} as

\aaa{\label{v-sol}
v(x, t) 
&=
-\frac{1}{2\pi} \int_{\p D_2}  \frac{e^{ikx+i\omega(k) t}}{i\omega(k)}\cdot \frac{k \left[k-\nu(k)\right]}{\Delta_2(k)} \left\{\left[\nu(k)+i\gamma\right] \mathcal F\{\psi\}(\omega) + i \nu(k) \left[\nu(k)+i\delta\right] \mathcal F\{\varphi\}(\omega(k)) \right\} dk
\\
& \quad -\frac{1}{2\pi} \int_{\p D_1}  \frac{e^{ikx-i\omega(k) t}}{i\omega(k)} \cdot \frac{k\left(k+\nu(k)\right)}{\Delta_1(k)} 
\left\{
\left[\nu(k)-i\gamma\right] \mathcal F\{\psi\}(-\omega(k)) - i \nu(k) \left(\nu(k)-i\delta\right)\mathcal F\{\varphi\}(-\omega(k))
\right\} dk.
\nn
}
Note that the Fourier transform in $t$ has replaced the ``tilde'' transforms present in \eqref{T-sol} due to the compact support condition on $\varphi, \psi$. 

\begin{theorem}\label{r-t}
Suppose $s\in \R$. Then, for each $t\in [0, T]$, the solution $v = S\big[0, 0, \varphi, \psi; 0\big]$ of the reduced initial-boundary value problem \eqref{lgbous-r} satisfies the space estimate
\eee{\label{r-se}
\no{\p_t^j v(t)}_{H_x^{s-2j}(0, \infty)}  
 \leq
  \left[c_j(s)+c_j(s, T) \sqrt{T+1} \big(1-e^{-(T+1)}\big)^{-1} \right] \left(\no{\varphi}_{H^{\frac{2s-1}{4}}(\R)} + \no{\psi}_{H^{\frac{2s-3}{4}}(\R)}\right),
  \quad j \in \N_0,
}
where the constant $c_j(s, T)$ remains bounded as $T\to 0^+$.
\end{theorem}

\begin{proof}
Observe that
\ddd{\label{v-solt}
\p_t^j v(x, t) 
&=
-\frac{1}{2\pi} \int_{\p D_2}  \frac{e^{ikx+i\omega(k) t}}{i\omega(k)}\cdot  \frac{k \left(k-\nu(k)\right)}{\Delta_2(k)} \left\{\left(\nu(k)+i\gamma\right) \mathcal F\{\psi^{(j)}\}(\omega) + i \nu(k) \left(\nu(k)+i\delta\right) \mathcal F\{\varphi^{(j)}\}(\omega(k)) \right\} dk
\\
& \quad -\frac{1}{2\pi} \int_{\p D_1}  \frac{e^{ikx-i\omega(k) t}}{i\omega(k)} \cdot   \frac{k\left(k+\nu(k)\right)}{\Delta_1(k)} 
\left\{
\left(\nu(k)-i\gamma\right) \mathcal F\{\psi^{(j)}\}(-\omega(k)) - i \nu(k) \left(\nu(k)-i\delta\right)\mathcal F\{\varphi^{(j)}\}(-\omega(k))
\right\} dk,
}
which is the same with formula \eqref{v-sol} for $v$ after replacing $\varphi, \psi$ by $\varphi^{(j)}$, $\psi^{(j)}$.
Therefore, we can deduce \eqref{r-se} for $j\in \N$ from the base case of $j=0$ with $s$ replaced by $s-2j$ and $\varphi, \psi$ replaced by $\varphi^{(j)}$, $\psi^{(j)}$. Thus, it suffices to establish \eqref{r-se} for $j=0$.

We write $2\pi v = I_{\varphi, 1} + I_{\varphi, 2} +  I_{\psi, 1} + I_{\psi, 2}$ where
\ddd{\label{iterms}
I_{\varphi, 1}(x, t) 
&=
 \int_{\p D_1}  \frac{e^{ikx-i\omega(k) t}}{i\omega(k)} \cdot \frac{k\left(k+\nu(k)\right)}{\Delta_1(k)} \, \nu(k) \left(\nu(k)-i\delta\right)\mathcal F\{\varphi\}(-\omega(k)) dk
\\
I_{\varphi, 2}(x, t) 
&=
\int_{\p D_2}  \frac{e^{ikx+i\omega(k) t}}{i\omega(k)}\cdot \frac{k \left[k-\nu(k)\right]}{\Delta_2(k)} \, \nu(k) \left[\nu(k)+i\delta\right] \mathcal F\{\varphi\}(\omega(k)) dk
\\
I_{\psi, 1}(x, t) 
&=
 \int_{\p D_1}  \frac{e^{ikx-i\omega(k) t}}{i\omega(k)} \cdot \frac{k\left(k+\nu(k)\right)}{\Delta_1(k)} 
\left[\nu(k)-i\gamma\right] \mathcal F\{\psi\}(-\omega(k))  dk
\\
I_{\psi, 2}(x, t) 
&=
\int_{\p D_2}  \frac{e^{ikx+i\omega(k) t}}{i\omega(k)}\cdot \frac{k \left[k-\nu(k)\right]}{\Delta_2(k)}  \left[\nu(k)+i\gamma\right] \mathcal F\{\psi\}(\omega(k))   dk
}
We will only provide details for $I_{\varphi, 1}$, since the remaining integrals can be handled in an entirely analogous manner.

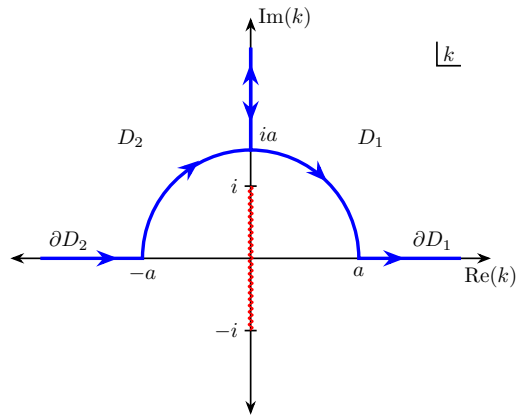
\begin{figure}[h!]
\begin{center}
\resizebox{7cm}{5.5cm}{
\begin{tikzpicture}[>=Stealth]

    \draw[<->,thick] (-4,0) -- (4,0) node[below]{Re($k$)};
    \draw[<->,thick] (0,-2.6) -- (0,4.0) node[right]{Im($k$)};

    \draw[thick] (-0.1, 1.2) -- (0.1, 1.2) node[left=5pt] {$i$}; 
     \draw[thick] (-0.1, -1.2) -- (0.1, -1.2) node[left=5pt] {$-i$}; 

      \draw[red, thick, branch cut] (0,1.2) -- (0,-1.2);

    \draw[blue, ultra thick, 
        decoration={markings, 
            mark=at position 0.2 with {\arrow{>}}, 
            mark=at position 0.5 with {\arrow{>}}, 
            mark=at position 0.85 with {\arrow{>}}}, 
        postaction={decorate}]
        (0.0, 3.5) -- (0.0, 1.8) node[above right]{\color{black}{$ia$}}
        arc[start angle=90, end angle=0, radius=1.8] node[below]{\color{black}{$a$}}
        -- (3.5, 0.0)
        node[above left, black] {$\partial D_1$};
         \draw[blue, ultra thick, 
        decoration={markings, 
            mark=at position 0.1 with {\arrow{<}}, 
            mark=at position 0.47 with {\arrow{<}}, 
            mark=at position 0.85 with {\arrow{<}}}, 
        postaction={decorate}]
        (0.0, 3.5) -- (0.0, 1.8)
        arc[start angle=90, end angle=180, radius=1.8] 
        node[below]{\color{black}{$-a$}}
        -- (-3.5, 0.0)
        node[above right, black] {$\partial D_2$};

  \node at (2, 2) {$D_1$};
  \node at (-2, 2) {$D_2$};
  
\begin{scope}[shift={(3.1,3.2)}]
        \draw[thick] (0,0.4) -- (0,0) -- (0.4,0);
        \node at (0.2, 0.2) {$k$};
    \end{scope}

\end{tikzpicture}
}
\end{center}
    \caption{The positively oriented boundaries $\p D_{1, 2}$   of the regions $D_{1, 2}$ for \eqref{v-solt} and \eqref{T-sol}.}
    \label{fig:0} 
\end{figure}

As depicted in Figure\eqref{fig:0}, the contour $\p D_1$ runs along the imaginary axis from $i\infty$ to $ia$ (where $a>0$ is large enough so as to avoid the zeros of $\Delta_1(k)$ and $\Delta_2(k)$ defined by \eqref{Del-2}), then becomes a quarter-circle from $ia$ to $a$, and finally aligns with the real axis along the portion $[a, \infty)$. We start with the portion of $I_{\varphi, 1}$ along the real axis, namely
\eee{\label{ifr-def}
I_{\varphi, 1}^r (x, t)
:=
\int_{a}^\infty  \frac{e^{ikx-i\omega(k) t}}{i\omega(k)} \cdot \frac{k\left(k+\nu(k)\right)}{\Delta_1(k)} \, \nu(k) \left(\nu(k)-i\delta\right)\mathcal F\{\varphi\}(-\omega(k)) dk.
}
This particular integral makes sense for all $x\in\mathbb R$ due to the fact that $|e^{ikx-i\omega(k)t}|=1$ for $k\in\mathbb R$. Thus, by the Fourier inversion theorem we have
\eee{
\mathcal F\{I_{\varphi, 1}^r\}(k, t)
=
2\pi \, \chi_{(a, \infty)}(k) \cdot \frac{e^{-i\omega(k) t}}{i\omega(k)} \cdot \frac{k\left(k+\nu(k)\right)}{\Delta_1(k)} \, \nu(k) \left(\nu(k)-i\delta\right)\mathcal F\{\varphi\}(-\omega(k)).
}
Then, by the Fourier transform characterization of $H^s(\mathbb R)$, we find
\aaa{
\no{I_{\varphi, 1}^r(t)}_{H_x^s(0, \infty)}^2
&\lesssim
\int_a^\infty \left(1+k^2\right)^s 
\left|
\frac{e^{-i\omega(k) t}}{i\omega(k)} \cdot \frac{k\left(k+\nu(k)\right)}{\Delta_1(k)} \, \nu(k) \left(\nu(k)-i\delta\right)\mathcal F\{\varphi\}(-\omega(k))
\right|^2 dk
\nn\\
&=
\int_a^\infty \left(1+k^2\right)^s 
\left(
\frac{\left|k+\nu(k)\right|   \left|\nu(k)-i\delta\right|}{\left|\Delta_1(k)\right|}\right)^2
\left| \mathcal F\{\varphi\}(-\omega(k))
\right|^2 dk.\label{redI-r1}
}
Here, we take 
\eee{
\Lambda_1^r(k) := \left(
\frac{\left|k+\nu(k)\right|   \left|\nu(k)-i\delta\right|}{\left|\Delta_1(k)\right|}\right)^2  = \frac{(1+2k^2)\left(1+k^2+\delta^2-2\delta\sqrt{1+k^2}\right)}{k^2\left[-\gamma + \left(1+k^2\right)^{\frac 12}\right]^2+\gamma^2\left[\delta-\left(1+k^2\right)^{\frac 12}\right]^2},
}
where we have simplified the function $\Lambda_1^r(k)$ using the definition of $\nu(k)$  as $k \in \R$ in this case.
Now, for fixed $\delta,\gamma \in \R$, we have $\Lambda_1^r(k) \to 2$ for $|k|\gg 1$, so we can take $\Lambda_1^r(k)\leq 3$ for sufficiently large $a$. Next, we let $k = \sqrt{\frac{-1 + \sqrt{1+4\tau^2}}{2}}$, which implies $\tau = -\omega(k)$ and 
$d\tau = -\omega'(k) dk = - \frac{1+2k^2}{\left(1+k^2\right)^{1/2}} \, dk.$
Therefore,
\aaa{
\no{I_{\varphi, 1}^r(t)}_{H_x^s(0, \infty)}^2
&\leq 
3 \int_{-\infty}^{-a \sqrt{1+a^2}} \left(1+\frac{-1 + \sqrt{1+4\tau^2}}{2}\right)^{s-\frac 12} 
 \left| \mathcal F\{\varphi\}(\tau)
\right|^2  \, d\tau
\nn\\
&\lesssim
\int_{-\infty}^{-a \sqrt{1+a^2}} \left(1+|\tau|\right)^{s-\frac 12} 
 \left| \mathcal F\{\varphi\}(\tau)
\right|^2  \, d\tau
\leq
\no{\varphi}_{H^{\frac{2s-1}{4}}(\mathbb R)}^2.\label{redI-r2}
}

Next, we consider the portion of $I_{\varphi, 1}$ along the imaginary axis. After parametrizing $k=i\lambda$, $\lambda \in [a, \infty)$, this reads
\eee{\label{im-phi}
I^i_{\varphi,1}(x,t) 
:=
\int_a^{\infty} \frac{e^{-\lambda x-i\omega(i\lambda)t}}{i\omega(i\lambda)}\cdot \frac{\lambda(i\lambda+\nu(i\lambda))}{\Delta_1(i\lambda)} \, \nu(i\lambda) \left(\nu(i\lambda)-i\delta\right) \mathcal{F}\{\varphi\} (-\omega(i\lambda))\,d\lambda
}
This integral only makes sense for $x \geq 0$ due to the presence of the exponential $e^{-\lambda x}$ for $\lambda>0$. Thus, instead of the Fourier characterization of the Sobolev norm, we employ the physical space norm.

For $s \in \N_0 := \mathbb{N} \cup \{0\}$, we have
\eee{\label{sn-def}
	\no{I^i_{\varphi,1}(t)}_{H^s_x(0,\infty)} = \sum_{j=0}^s \no{\p^j_x I^i_{\varphi,1}(t)}_{L^2_x(0,\infty)}.
}
For any $j \in \{0,1,\ldots,s\}$, 
\aaa{
\no{\p_x^j I^i_{\varphi,1}(t)}^2_{L^2_x(0,\infty)}
&=
\int_0^\infty \left|  \int_a^{\infty} (-\lambda)^j \frac{e^{-\lambda x-i\omega(i\lambda)t}}{\omega(i\lambda)}\cdot \frac{\lambda(i\lambda+\nu(i\lambda))}{\Delta_1(i\lambda)} \, \nu(i\lambda)(\nu(i\lambda)-i\delta) \mathcal{F}\{\varphi\} (-\omega(i\lambda))\,d\lambda \right|^2dx
\nn\\
&\equiv
\no{\mathcal L\{f\}}_{L_x^2(0, \infty)}^2
 \label{L2}
}
where $\mathcal L\{f\}(x) := \displaystyle \int_0^\infty e^{-\lambda x} f(\lambda) d\lambda$ is the Laplace transform of $f(\lambda)$ and, in our case,
\eee{
f(\lambda) = f(\lambda, t)  = \chi_{(a, \infty)}(\lambda)  \cdot (-\lambda)^j \frac{e^{-i\omega(i\lambda)t}}{\omega(i\lambda)}\cdot \frac{\lambda(i\lambda+\nu(i\lambda))}{\Delta_1(i\lambda)} \, \nu(i\lambda)(\nu(i\lambda)-i\delta) \mathcal{F}\{\varphi\} (-\omega(i\lambda)).
}
\begin{lemma}[\cite{h1929}]\label{hardy-l}
The Laplace transform $f \mapsto \mathcal L\{f\}$  is bounded from $L^2(0, \infty)$ to $L^2(0, \infty)$ with 
\eee{
\no{\mathcal L\{f\}}_{L^2(0, \infty)} \leq \sqrt \pi \no{f}_{L^2(0, \infty)}. 
}
\end{lemma}
Using the above result, we obtain
\eee{
\no{\p_x^j I^i_{\varphi,1}(t)}^2_{L^2_x(0,\infty)}
 \lesssim  \int_a^{\infty} \left|  (-\lambda)^j \frac{e^{-i\omega(i\lambda)t}}{\omega(i\lambda)}\cdot \frac{\lambda(i\lambda+\nu(i\lambda))}{\Delta_1(i\lambda)} \, \nu(i\lambda)(\nu(i\lambda)-i\delta) \mathcal{F}\{\varphi\} (-\omega(i\lambda))\right|^2 d\lambda. \label{L2-norm}
}
Since $\omega(i\lambda) = -\lambda(\lambda^2-1)^{\frac 12} \in \R$, the exponential $e^{-i\omega(i\lambda)t}$ is oscillatory implying $|e^{-i\omega(i\lambda)t}|=1$. Moreover, $\nu(i\lambda) =  -(\lambda^2-1)^{\frac 12}$, $|i\lambda+\nu(i\lambda)| = (2\lambda^2-1)^{\frac 12}$, $|\nu(i\lambda)-i\delta| = (\lambda^2-1+\delta^2)^{\frac 12}$ and $\Delta_1(i\lambda) = (\lambda^2-1)^{\frac 12} (\gamma-\lambda)+i\gamma(\delta-\lambda)$. Thus, 
\eee{\label{325}
\no{\p_x^j I^i_{\varphi,1}(t)}^2_{L^2_x(0,\infty)}  \lesssim \int_a^{\infty}  (\lambda^2)^{j} \Lambda_1^i(\lambda) \, |\mathcal{F}\{\varphi\} (-\omega(i\lambda))|^2 d\lambda 
}
where 
\eee{
\Lambda_1^i(\lambda) := \frac{(2\lambda^2-1)(\lambda^2-1+\delta^2)}{(\lambda^2-1)(\gamma-\lambda)^2+\gamma^2(\delta-\lambda)^2}.
}
Now, for fixed $\delta,\gamma \in \R$, we have $\Lambda_1^i(\lambda) \to 2$ for $|\lambda|\gg 1$, so we can take $\Lambda_1^i(\lambda)\leq 3$ for sufficiently large $a$. Next, we let $\lambda =  \sqrt{\frac{1 + \sqrt{1+4\tau^2}}{2}}$, which implies $\tau=-\omega(i\lambda)$. In turn, 
$
\lambda^2-1 
= \frac{2\tau^2}{1+\sqrt{1+4\tau^2}}$ 
so 
$\sqrt{1+4\tau^2}(1+\sqrt{1+4\tau^2})^{\frac12} d\lambda =  \sqrt2\tau \,d\tau$.
Therefore, 
\aaa{
\no{\p_x^j I^i_{\varphi,1}(t)}^2_{L^2_x(0,\infty)}  
& \lesssim \int_{a\sqrt{a^2-1}}^{\infty} \left( \frac{1+\sqrt{1+4\tau^2}}{2}\right)^j  \frac{\sqrt2\tau}{\sqrt{1+4\tau^2}(1+\sqrt{1+4\tau^2})^{\frac12}}|\mathcal{F}\{\varphi\}(\tau)|^2d\tau 
\nn\\
& \lesssim \int_{a\sqrt{a^2-1}}^{\infty} \left(1+\tau^2\right)^{\frac{2j-1}{4}}|\mathcal{F}\{\varphi\}(\tau)|^2d\tau
 \leq \no{\varphi}_{H^{\frac{2j-1}{4}}(\R)}^2 \label{rvp-est}
} 
which in light of \eqref{sn-def} yields 
\eee{\label{int-est}
	\no{I^i_{\varphi,1}(t)}_{H^s_x(0,\infty)} \lesssim \no{\varphi}_{H^{\frac{2s-1}{4}}(\R)}, \quad s \in \N_0, \ t\in [0, T].
}

Now, we will establish the estimate \eqref{int-est} for all $s \geq 0$. In that case,
\eee{\label{223}
\no{f}_{H^s(0,\infty)} = \sum_{j=0}^{\sigma}\no{\p_x^j f}_{L^2(0,\infty)} +\no{\p^{\sigma} f}_{\beta},
\quad
\sigma = \floor{s}, \ \beta = s - \sigma \in [0, 1),
}
where the  fractional seminorm $\no{\cdot}_\beta$ is $0$ for $\beta=0$ and is otherwise given by
\eee{ 
\no{f}^2_{\beta} = \int_0^\infty \int_0^\infty \frac{|f(x+z)-f(x)|^2}{z^{1+2\beta}}\,dzdx.
}
The sum on the right side of \eqref{223} can be readily estimated via \eqref{int-est}. Concerning the fractional seminorm, we compute
\eee{
\left| \p_x^{\sigma} \left( I^i_{\varphi,1}(x+z,t)-I^i_{\varphi,1}(x,t)\right)\right| 
\leq
 \int_a^\infty \frac{\lambda^\sigma e^{-\lambda x}}{|\omega(i\lambda)|}\cdot \frac{\lambda|i\lambda+\nu(i\lambda)|}{|\Delta_1(i\lambda)|} \nu(i\lambda)^2 |\nu(i\lambda)-i\delta|\cdot |\mathcal{F}\{\varphi\}(-\omega(i\lambda))|(1-e^{-\lambda z})\,d\lambda
}
and so, using Lemma \ref{hardy-l} between the $x$ and $\lambda$ integrals, we have
\eee{
\no{\p^{\sigma}_x I^i_{\varphi,1}(t)}^2_{\beta} 
\lesssim \int_a^\infty \frac{\lambda^{2\sigma+2}|i\lambda+\nu(i\lambda)|^2}{|\omega(i\lambda)|^2|\Delta_1(i\lambda)|^2}  \nu(i\lambda)^2 |\nu(i\lambda)-i\delta|^2\cdot |\mathcal{F}\{\varphi\}(-\omega(i\lambda))|^2 \left( \int_0^\infty \frac{(1-e^{-\lambda z})^2}{ z^{1+2\beta}}\,dz\right) d\lambda. \label{frac-norm}
}

Letting $\lambda z=\zeta$ in the improper integral with respect to $z$, we have
\eee{
\int_0^\infty \frac{(1-e^{-\lambda z})^2}{ z^{1+2\beta}}\,dz
=
\lambda^{2\beta}
\left(\int_0^1 \frac{(1-e^{-\zeta})^2}{\zeta^{1+2\beta}}\,d\zeta 
+
\int_1^\infty \frac{(1-e^{-\zeta})^2}{\zeta^{1+2\beta}}\,d\zeta\right).
}
For $\zeta \geq 1$, $(1-e^{-\zeta})^2 \leq 4$ so 
\eee{
\int_1^\infty \frac{(1-e^{-\zeta})^2}{\zeta^{1+2\beta}}\,d\zeta \leq 4 \int_1^\infty \frac{1}{\zeta^{1+2\beta}}\,d\zeta =\frac 2\beta.
}
For $\zeta \in [0, 1]$, by the Mean Value Theorem for the function $1-e^{-\zeta}$ on $[0,\zeta]$, there exists $c\in(0,\zeta)\subset (0,1)$ such that $1-e^{-\zeta} = e^{-c}\zeta \leq \zeta$. 
Therefore, 
\eee{
\int_0^1 \frac{(1-e^{-\zeta})^2}{\zeta^{1+2\beta}}\,d\zeta \leq \int_0^1 \frac{\zeta^2}{\zeta^{1+2\beta}}\,d\zeta = \frac{2}{2-2\beta}.
}
Overall, the $z$ integral in \eqref{frac-norm} converges, thus
\eee{
\no{\p^{\sigma}_x I^i_{\varphi,1}(t)}^2_{\beta} 
\lesssim
   \int_a^\infty  \lambda^{2s} \, \frac{\lambda^{2}|i\lambda+\nu(i\lambda)|^2}{|\omega(i\lambda)|^2|\Delta_1(i\lambda)|^2}  \nu(i\lambda)^2 |\nu(i\lambda)-i\delta|^2\cdot |\mathcal{F}\{\varphi\}(-\omega(i\lambda))|^2 d\lambda 
}
after recalling that $s = \sigma + \beta$. The above integral is comparable to the one in \eqref{L2-norm}, thus proceeding analogously we have 
\eee{\label{fn}
\no{\p^{\sigma}_x I^i_{\varphi,1}(t)}^2_{\beta} 
\lesssim \int_a^\infty \left(1+\tau^2\right)^{\frac{2s-1}4} |\mathcal{F}\{\varphi\}(\tau)|^2d\tau 
\leq
 \no{\varphi}^2_{H^{\frac{2s-1}{4}}(\R)}
}
Combining the fractional estimate \eqref{fn} with \eqref{223} and the integer estimate \eqref{int-est}, we conclude that 
\eee{\label{340}
\no{I^i_{\varphi,1}(t)}_{H^s_x(0,\infty)} \lesssim  \no{\varphi}_{H^{\frac{2s-1}{4}}(\R)}, \quad s\geq 0, \ t\in [0, T].
}

Next, we consider the case of $s \in \mathbb{Z}^{-}$, in which we define the Sobolev norm via
\eee{\label{NegS-def}
	\no{I_{\varphi,1}^i(t)}_{H_x^s(0,\infty)} 
	= 
	\sup\left\{\left|\langle I_{\varphi,1}^i(t), f \rangle \right|: f \in H_0^{-s}(0,\infty) \text{ with } \no{f}_{H^{-s}(0, \infty)}=1\right\},
}
where $\langle I_{\varphi,1}^i(t), f \rangle := \int_0^\infty I_{\varphi,1}^i(x, t) \overline{f(x)}\,dx$ and $H_0^{-s}(0, \infty)$ is the subset of $H^{-s}(0, \infty)$ whose elements have vanishing Sobolev traces at $0$ up to derivative order less than $-s-\frac 12$ (see \cite[Theorem 11.5]{lm1972}).
Thus, letting 
\eee{
A(\lambda,t)
=
\frac{e^{-i\omega(i\lambda)t}}{\omega(i\lambda)}\cdot 				\frac{i\lambda(i\lambda+\nu(i\lambda))}{\Delta_1(i\lambda)} \, 
	\nu(i\lambda)(\nu(i\lambda)-i\delta) \mathcal{F}\{\varphi\} (-\omega(i\lambda))
}
and integrating by parts while taking into account that  $\p^j f(0)=0$ for all $0 \leq j < |s|- \frac12$, we obtain
\aaa{\label{imag-dual}
	\langle  I_{\varphi,1}^i(t), f \rangle 
	& = 
	\int_a^\infty 
	A(\lambda,t) 
 	\int_0^{\infty} 
	e^{-\lambda x} \overline{f(x)}\,dx \, d\lambda
	\nn\\
	& =
	\int_a^\infty A(\lambda,t) \int_0^\infty 
	\frac{1}{\lambda^{|s|}} e^{-\lambda x} 
	\, \overline{\p_x^{|s|} f(x)} \, dx\, d\lambda 
	\nn\\
	& \leq
	\left(   
	\int_0^\infty 
	\left\lvert 
	\int_a^\infty \lambda^s e^{-\lambda x} A(\lambda,t)\,d\lambda 
	\right\rvert^2 dx
	\right)^\frac12 
}
where for the last step we have used the Cauchy-Schwarz inequality and the fact that $\no{f}_{H^{|s|}(0, \infty)}=1$.
The right side is analogous to the one in \eqref{L2} when $j = s$, so we conclude that
\eee{\label{if1n}
\no{I_{\varphi,1}^i(t)}_{H_x^s(0,\infty)}
	\leq c(s,T)
	\no{\varphi}_{H^{\frac{2s-1}4}(\R)}, \quad s \in \mathbb Z^-, \ t\in [0, T].
}
In fact, we can extend the validity of \eqref{if1n} to all $s<0$ via interpolation. This completes the estimation of $I_{\varphi,1}^i$.

Finally, we consider the integral along the portion of $\p D_1$ that consists of the negatively oriented quarter-circle $C:=\left\{ k = a e^{i\theta}, 0 \leq \theta \leq \frac \pi 2\right\}$, namely  
\aaa{
I^C_{\varphi, 1}(x, t) 
&:=
 \int_{C}  \frac{e^{ikx-i\omega(k) t}}{i\omega(k)} \cdot \frac{k\left(k+\nu(k)\right)}{\Delta_1(k)} \, \nu(k) \left(\nu(k)-i\delta\right)\mathcal F\{\varphi\}(-\omega(k)) dk
\nn\\
&= -\int_0^{\frac\pi2} \frac{e^{iae^{i\theta}x}\cdot e^{-i\omega(ae^{i\theta})t}}{i\omega(ae^{i\theta})}\cdot \frac{ae^{i\theta}\left(ae^{i\theta}+\nu(ae^{i\theta})\right)}{\Delta_1(ae^{i\theta})}\nu(ae^{i\theta})\left(\nu(ae^{i\theta})-i\delta\right)\mathcal{F}\{\varphi\}(-\omega(ae^{i\theta}))iae^{i\theta}\,d\theta. 
\label{234}
}
For any $j\in\N_0$,  
\ddd{
 \no{\p_x^j I^C_{\varphi,1}(t)}^2_{L^2_x(0,\infty)} 
\leq \int_0^\infty \biggl( \int_0^\frac\pi2 &a^j e^{-ax\sin\theta} \cdot \frac{e^{\text{Im}\left(\omega(ae^{i\theta})\right)t}}{|\omega(ae^{i\theta})|}\cdot \frac{a\left(a+\sqrt{1+a^2}\right)}{|\Delta_1(ae^{i\theta})|}\\
&
\cdot  (1+a^2)^{\frac 12}(\sqrt{1+a^2}+|\delta|)\left|\mathcal{F}\{\varphi\}\left(-\omega(ae^{i\theta})\right)\right|a\,d\theta\biggr)^2 dx. 
}
By the triangle inequality, $a\sqrt{a^2-1} \leq |\omega(ae^{i\theta})| \leq a\sqrt{a^2+1}$, so using the fact that $\text{Im}(\omega) \leq {|\omega|}$ and $\Delta_1(ae^{i\theta})\simeq a^2$ for sufficiently large $a$, we have
\eee{\label{phi-C}
 \no{\p_x^j I^C_{\varphi,1}(t)}^2_{L^2_x(0,\infty)} 
\lesssim \left(a^{j+1} e^{a\sqrt{1+a^2}T}\right)^2 \int_0^\infty \left( \int_0^\frac\pi2 e^{-ax \sin\theta}\left|\mathcal{F}\{\varphi\}\left(-\omega(ae^{i\theta})\right)\right|\,d\theta\right)^2 dx.
}
Although a direct estimation of the above integral does lead to the desired Sobolev space $H^{\frac{2s-1}{4}}$, it imposes the restriction $s \geq \frac12$. To bypass this issue, we instead proceed as follows.

Let $\eta(t) = \chi_{[0, T+1]}(t) \, e^{-t}$ and note that 
 \eee{\label{eta}
 \mathcal{F}\{\eta\}(\tau)= \frac{1-e^{-(T+1)(1+i\tau)}}{1+i\tau}.
} 
Since $\text{Im}(\omega) >0$ along $C$ for sufficiently large $a$, 
\eee{
1-e^{-(T+1)} \leq \big| 1-e^{-(T+1)(1-i\omega(ae^{i\theta})}\big| \leq 1+e^{-(T+1)}, 
\quad
a\sqrt{a^2-1}-1  \leq \left|1-i\omega(ae^{i\theta}) \right| \leq a\sqrt{a^2-1}+1 
}
so for large enough $a$ we deduce  
\eee{\label{LbUb}
	\frac{1-e^{-(T+1)}}{a^2} \lesssim |\mathcal{F}\{\eta\}(-\omega(ae^{i\theta}))| \lesssim \frac{1+e^{-(T+1)}}{a^2}.
}
Hence, for any $d \in \N$, 
\aaa{
\left| \mathcal{F}\{\varphi\}(-\omega(ae^{i\theta}))\right| 
& = 
\frac{1}{\left| \mathcal{F}\{\eta\}(-\omega(ae^{i\theta})) \right|^d} \left| \left[ \mathcal{F}\{\eta\}(-\omega(ae^{i\theta})) \right]^d\mathcal{F}\{\varphi\}(-\omega(ae^{i\theta}))\right| \nn\\
& \lesssim 
\frac{a^{2d}}{(1-e^{-(T+1)})^d} \left| \left[ \mathcal{F}\{\eta\}(-\omega(ae^{i\theta})) \right]^d\mathcal{F}\{\varphi\}(-\omega(ae^{i\theta}))\right|.
\label{243}
}
Observe that $\left[ \mathcal{F}\{\eta\}(-\omega(ae^{i\theta})) \right]^d\mathcal{F}\{\varphi\}(-\omega(ae^{i\theta}))  = \mathcal{F}\{ \eta^{*d}\ast \varphi \}
(-\omega(ae^{i\theta}))$, where $\eta^{*d}$ denotes the $d$-fold convolution of $\eta$. Furthermore, by the support of $\varphi$ and $\eta$,  $\text{supp}\{\eta^{*d}\ast \varphi\} \subseteq \left[0,(d+1)T+d\right]$. Therefore, using once again the fact that $\text{Im}(\omega) \geq 0$ on $C$ and the Cauchy-Schwarz inequality, we find
\aaa{
\left| \left[ \mathcal{F}\{\eta\}(-\omega(ae^{i\theta})) \right]^d\mathcal{F}\{\varphi\}(-\omega(ae^{i\theta})) \right|
&= 
\left| \int_0^{(d+1)T+d} e^{i\omega(ae^{i\theta})t} \left(\eta^{*d}\ast\varphi\right)(t) \,dt\right| 
\nn\\
& \leq  \sqrt{(d+1)T +d} \no{\eta^{*d}\ast\varphi}_{L^2(\R)} 
\nn\\
&\simeq
\sqrt{(d+1)T+d} 
\left( \int_\R \left| [\mathcal F \{\eta\}(\tau)]^d \mathcal{F}\{\varphi\}(\tau)\right|^2d\tau \right)^{\frac12}
}
by Plancherel's theorem. Recalling \eqref{eta} and noting that, for $\tau \in \R$, we have $\left|1-e^{-T(1+i\tau)}\right|\leq 1+e^{-(T+1)}$ and $|1+i\tau|^{2d}= \left(1+\tau^2\right)^d$, we further obtain
\eee{
\left| \left[ \mathcal{F}\{\eta\}(-\omega(ae^{i\theta})) \right]^d\mathcal{F}\{\varphi\}(-\omega(ae^{i\theta})) \right|
\lesssim
\sqrt{(d+1) T+d}\, \big(1+e^{-(T+1)}\big)^d \no{\varphi}_{H^{-d}(\R)},
\quad d\in\N.\label{Conv2}
 }
Combining this bound with \eqref{243} and \eqref{phi-C} yields
\eee{
 \no{\p_x^j I^C_{\varphi,1}(t)}_{L^2_x(0,\infty)} 
 \lesssim
 \frac{a^{j+1+2d}e^{a\sqrt{a^2+1}T}}{(1-e^{-(T+1)})^{d}}\, \sqrt{(d+1)T+d}\, \big(1+e^{-(T+1)}\big)^d \no{\varphi}_{H^{-d}(\R)} \sqrt{I(a)}
}
where
\eee{
I(a) := \int_0^\infty \left( \int_0^{\frac{\pi}{2}} e^{-ax\sin \theta }d\theta\right)^2dx
=
\int_0^1 \left( \int_0^{\frac{\pi}{2}} e^{-ax\sin \theta}d\theta\right)^2dx
+
\int_1^\infty \left( \int_0^{\frac{\pi}{2}} e^{-ax\sin \theta}d\theta\right)^2dx.
}
The first of the integrals on the right side is finite since $e^{-ax\sin \theta}\leq 1$ in view of the fact that $a,x \geq 0$ and $\sin \theta \geq 0 $ for $\theta \in \left[0,\frac{\pi}2\right]$. The second integral is also finite after using the inequality $\sin \theta \geq \frac{2\theta}{\pi}$, $0 \leq \theta \leq \frac{\pi}{2}$. 
Hence, $I(a) < \infty$ and 
\eee{
 \no{\p_x^j I^C_{\varphi,1}(t)}_{L^2_x(0,\infty)} 
\lesssim
 c(a, d, j, T)   \big(1-e^{-(T+1)}\big)^{-d}  \no{\varphi}_{H^{-d}(\R)}, \quad j \in \N_0, \ d\in\N, \ t\in [0, T].
 }
 When $s\geq 0$, any choice of $d \in \mathbb{N}$ satisfies $d\geq \frac{1-2s}4$. On the other hand, for $s < 0$, choosing $d = \left\lceil \frac{1-2s}{4} \right\rceil$ ensures $d\geq \frac{1-2s}4$. Therefore, we conclude that
\eee{
 \no{\p_x^j I^C_{\varphi,1}(t)}_{L^2_x(0,\infty)} 
\lesssim
c(a, s, j, T) \big(1-e^{-(T+1)}\big)^{-\left\lceil \frac{1-2s}{4} \right\rceil} \no{\varphi}_{H^{\frac{2s-1}{4}}(\R)}, \quad s \in \R, \ j \in \N_0, \ t\in [0, T],
\label{ic-est}
}
This result allows us to control \textit{any} Sobolev norm of $I^C_{\varphi,1}(t)$ by the $H^{\frac{2s-1}{4}}(\R)$ norm of $\varphi$ for any $s\in\R$.
\end{proof}

\section{Estimates for the Cauchy problem}
\label{cauchy-s}

In the previous section, we obtain estimates on the reduced initial-boundary value problem \eqref{lgbous-r}. In order to employ these estimates for the full linear initial-boundary value problem \eqref{lgbous}, we first establish necessary estimates for the linear Cauchy problem. We begin with the homogeneous case and then proceed to the forced problem.

\subsection{Homogeneous linear Cauchy problem}

Consider the homogeneous Cauchy problem
\ddd{\label{ivp}
&U_{tt} - U_{xx} + U_{xxxx} = 0, \quad  x \in \mathbb R, \ t\in (0, T), 
\\
&U(x,0) = U_0(x) \in H^s(\R), 
\quad
U_t(x,0) = U_1(x) \in H^{s-2}(\R),
}
with solution given by
\ddd{\label{ivp-sol}
U(x,t) & = S\big[U_0, U_1; 0\big](x, t)
\\
&:= \frac1{2\pi} \int_\R \frac{e^{ikx}}{2i\omega} 
\left\{ 
e^{i\omega t} \left[ \mathcal F \{U_1\}(k)+i\omega\mathcal F\{U_0\}(k)\right] - e^{-i\omega t} \left[ \mathcal F \{U_1\}(k)-i\omega\mathcal F\{U_0\}(k)\right]
\right\}dk.
}

\begin{theorem}[Homogeneous Cauchy estimates]\label{ivp-t}
For any $s\in \R$, the solution $U = S\big[U_0, U_1; 0\big]$ to the homogeneous linear Cauchy problem \eqref{ivp} admits the space and time estimates
\aaa{\label{ivp-se-t}
&\no{U(t)}_{H^s_x(\R)} 
+
\no{\p_t U(t)}_{H^{s-2}_x(\R)} 
\leq
c(T) \left(\no{U_0}_{H^{s}(\R)}  +  \no{U_1}_{H^{s-2}(\R)} \right),
\quad t\in [0, T], 
\\
&\no{\p_x^j U(x)}_{H_t^{\frac{2s-2j+1}{4}}(0, T)} 
\leq
c_j(s, T) 
\left(
 \no{U_0}_{H^s(\R)} + \no{U_1}_{H^{s-2}(\R)}\right),
 \quad
 j \in \N_0, \ x\in \R,
 \label{ivp-te-t}
}
where $c(T), c_j(s, T)>0$ are constants that remain bounded as $T\to 0^+$. 
\end{theorem}

\begin{proof}
We first establish the space estimate \eqref{ivp-se-t}. 
Noting that
\eee{
\mathcal F \{U\} (k,t)=  \frac{1}{2i\omega} 
\left\{ 
e^{i\omega t} \left[ \mathcal F \{U_1\}(k)+i\omega\mathcal F\{U_0\}(k)\right] - e^{-i\omega t} \left[ \mathcal F \{U_1\}(k)-i\omega\mathcal F\{U_0\}(k)\right]
\right\},
}
we have
\aaa{
\no{U(t)}^2_{H^s_x(\R)} 
&= 
\int_\R   \frac{ \left(1+k^2\right)^s  }{4\omega^2} \left| e^{i\omega t} \left[ \mathcal F \{U_1\}(k)+i\omega\mathcal F\{U_0\}(k)\right] - e^{-i\omega t} \left[ \mathcal F \{U_1\}(k)-i\omega\mathcal F\{U_0\}(k)\right]\right|^2 dk 
\nn\\
&\leq
\int_\R  \left(1+k^2\right)^s  \frac{\sin^2(\omega t)}{\omega^2}  \left|\mathcal F \{U_1\}(k)\right|^2 dk + 
\int_\R  \left(1+k^2\right)^s \cos^2(\omega t) \left|\mathcal F \{U_0\}(k)\right|^2 dk \label{U-Sobnorm0}.
}
Since $\omega \in \R$ for $k\in \R$, we have $\cos^2(\omega t) \leq 1$ so
\eee{
\int_\R  \left(1+k^2\right)^s \cos^2(\omega t) \left|\mathcal F \{U_0\}(k)\right|^2 dk \leq \int_\R  \left(1+k^2\right)^s \left|\mathcal F \{U_0\}(k)\right|^2 dk = \no{U_0}^2_{H^s(\R)}.
}
Further, breaking the first integral in \eqref{U-Sobnorm0} for $k$ near and away from $0$, we have 
\eee{
\int_\R  \left(1+k^2\right)^s  \frac{\sin^2(\omega t)}{\omega^2}  \left|\mathcal F \{U_1\}(k)\right|^2 dk = \left( \int_{|k|>1} + \int_{|k|<1} \right) \left(1+k^2\right)^s  \frac{\sin^2(\omega t)}{\omega^2}  \left|\mathcal F \{U_1\}(k)\right|^2 dk
}
For $|k|>1$, we have $\omega^2=k^2\left(1+k^2\right)\geq \frac12 \left(1+k^2\right)^2$ and $\sin^2(\omega t)\leq 1$, thus 
\eee{
\int_{|k|>1}  \left(1+k^2\right)^s \frac{\sin^2(\omega t)}{\omega^2}  \left|\mathcal F \{U_1\}(k)\right|^2 dk 
\leq
2  \int_{|k|>1}  \left(1+k^2\right)^{s-2}   \left|\mathcal F \{U_1\}(k)\right|^2 \leq 
2 \no{U_1}^2_{H^{s-2}(\R)}.
}
For $|k|<1$, we use the fact that $\frac{\sin^2(\omega t)}{\omega^2} \leq t^2$ and $1 \leq 4 \left(1+k^2\right)^{-2}$ to deduce
\eee{
\int_{|k|<1}  \left(1+k^2\right)^s  \frac{\sin^2(\omega t)}{\omega^2}  \left|\mathcal F \{U_1\}(k)\right|^2 dk
 \leq 4t^2 \int_{|k|<1}  \left(1+k^2\right)^{s-2}   \left|\mathcal F \{U_1\}(k)\right|^2 dk  \leq 4T^2 \no{U_1}^2_{H^{s-2}(\R)}.
}
Therefore,  \eqref{U-Sobnorm0} yields
\eee{\label{ivp-se}
\no{U(t)}_{H^s_x(\R)} 
\leq
\no{U_0}_{H^{s}(\R)} + 2 \sqrt{1+T^2} \no{U_1}_{H^{s-2}(\R)}, \quad  s\in\R, \ t\in [0, T].
}
The norm $\no{\p_t U(t)}_{H^{s-2}_x(\R)}$ can be estimated in a similar way --- in fact, as $\omega$ is no longer present in the denominator of the corresponding integrand, there is no need to break near and away from the origin. Overall, this yields the space estimate \eqref{ivp-se-t}.

We proceed to the time estimate \eqref{ivp-te-t}. Differentiating \eqref{ivp-sol}, we have
\aaa{\label{ivp-solx}
\p_x^j U(x,t) 
&= \frac1{2\pi} \int_\R \frac{e^{ikx}}{2i\omega} 
\left\{ 
e^{i\omega t} \left[ \mathcal F \{U_1^{(j)}\}(k)+i\omega\mathcal F\{U_0^{(j)}\}(k)\right] - e^{-i\omega t} \left[ \mathcal F \{U_1^{(j)}\}(k)-i\omega\mathcal F\{U_0^{(j)}\}(k)\right]
\right\}dk
\nn\\
&=
S\big[U_0^{(j)}, U_1^{(j)}; 0\big](x, t).
}
Thus, it suffices to prove \eqref{ivp-te-t} for $j=0$, as that base case can then be employed with $s$ replaced by $s-j$ and with $U_0, U_1$ replaced by $U_0^{(j)}, U_1^{(j)}$ to establish the general case of $j\in\N$. 
We write $U$ as the sum of four integrals 
\eee{
U = I^+_0 + I^-_0 + I^+_1 - I^-_1 
}
where 
\eee{
I_0^{\pm}(x,t) = \frac1{4\pi} \int_\R e^{ikx\pm i\omega t} \mathcal{F}\{U_0\}(k)\,dk,
\quad
I_1^{\pm}(x,t) = \frac1{2\pi} \int_\R \frac{e^{ikx\pm i\omega t}}{2i\omega} \mathcal{F}\{U_1\}(k)\,dk.
}

The terms $I_0^\pm$ can be analyzed in the same way, so we focus on 
$I_0^+$. We write
\eee{
I_0^+(x,t) = \frac{1}{4\pi}
\biggl(\underbrace{ \int_{-\infty}^{-1}}_{J_0} \enspace  
+ \underbrace{\int_{-1}^1}_{K_0} \enspace + \underbrace{\int_1^\infty}_{L_0} \biggr)
e^{ikx+i\omega t} \mathcal{F}\{U_0\}(k)\,dk.
}
The integrals $J_0$ and $L_0$ are similar, so we only provide the details for the former one. Letting 
\eee{\label{cov}
k = -\xi(\tau), \quad \xi(\tau) := \sqrt{\frac{-1 + \sqrt{1+4\tau^2}}{2}} \geq 0,
}
which implies $\tau \in [\sqrt 2, \infty)$ and $\omega = \tau$ (note that, according to the choice of branch cut for the square root $\left(1+k^2\right)^{\frac 12}$,  $\omega \geq 0$ for $k\in\mathbb R$), we have
\eee{
J_0(x, t) = \frac{1}{4\pi} \int_{\sqrt2}^\infty  e^{-i\xi(\tau)x} e^{i\tau t} \mathcal F\{U_0\}(-\xi(\tau)) \xi'(\tau)\,d\tau.
}
In turn, using also the bound  $|\xi'(\tau)| \lesssim \frac{|\tau|}{\sqrt{1+\tau^2}\left(1+\tau^2\right)^\frac14} \simeq \left(1+\tau^2\right)^{-\frac14}$ for $|\tau| \geq \sqrt 2$, we obtain
\aaa{
\no{J_0(x)}^2_{H_t^m(\R)} 
&\simeq 
\int_{\sqrt2}^\infty \left(1+\tau^2\right)^m  |\mathcal F\{U_0\}(-\xi(\tau))|^2 |\xi'(\tau)|^2\,d\tau
\nn\\
&\lesssim 
 \int_{\sqrt2}^\infty \left(1+\tau^2\right)^{m-\frac12}  |\mathcal F\{U_0\}(-\xi(\tau))|^2 d\tau.
}
Switching back to $k$ via letting $k= -\xi(\tau)$, which implies $\tau = \omega$, we find
\aaa{
\no{J_0(x)}^2_{H_t^m(\R)} 
 & \lesssim 
 \int_{-\infty}^{-1} \left(1+\omega^2\right)^{m-\frac12}  \left|\mathcal F\{U_0\}(k)\right|^2   \frac{1+2k^2}{\left(1+k^2\right)^{\frac12}} dk \nn\\
& \leq
2 \int_{-\infty}^{-1} \left(1+k^2+k^4\right)^{m-\frac12} \left(1+k^2\right)^{\frac12} \left|\mathcal F\{U_0\}(k)\right|^2 dk
\nn\\ 
 &\simeq
 \int_{-\infty}^{-1} \left(1+k^2\right)^{2m-\frac12}  \left|\mathcal F\{U_0\}(k)\right|^2 dk  \leq 
 \no{U_0}^2_{H^{2m-\frac12}(\R)}. \label{J0-est}
 }
Regarding the integral $K_0$, for any $\ell \in \N_0$, using the Cauchy-Schwarz inequality we have
\aaa{
\no{\p_t^\ell K_0(x)}_{L^2_t(0,T)}^2  
&=
\int_0^T \left| \int_{-1}^1 e^{ikx+i\omega t}  (i\omega)^\ell  \mathcal F\{U_0\}(k)\,dk\right|^2dt 
 \nn \\
&\leq
T \left( \int_{-1}^1 \omega^{2\ell}  \left(1+k^2\right)^{-s} dk \right)
\left(\int_0^1 \left(1+k^2\right)^s  \left|\mathcal F\{U_0\}(k)\right|^2 dk \right) 
\leq 
c_{s,\ell}^2 \, T  \no{U_0}^2_{H^s(\R)}. \label{Const}
}
Thus, for any $m \leq 0$, using \eqref{Const} with $\ell = 0$ yields
\eee{
\no{K_0(x)}_{H_t^m(0, T)} \leq \no{K_0(x)}_{L_t^2(0, T)} \leq  c_s \, T  \no{U_0}_{H^s(\R)}. 
}
Furthermore, for any $m>0$, \eqref{Const} and the analogue of definition \eqref{sn-def} with $(0, T+1)$ instead of $(0, \infty)$ imply
\eee{
\no{K_0(x)}_{H_t^m(0, T)} 
\leq  \no{K_0(x)}_{H_t^{\ceil{m}}(0, T+1)} 
= 
c_{s, T+1} \sum_{\ell=0}^{\ceil m} \no{\p_t^\ell K_0(x)}_{L_t^2(0, T+1)} 
\leq  
c_{s, T+1} \, c_{s, m} \, \sqrt{T+1}  \no{U_0}_{H^s(\R)}.
}
Here, we note that we have used the physical space equivalent Sobolev norm and, importantly, the associated equivalence constant $c_{s, T+1}$ remains bounded as $T\to 0^+$ since in that limit $(0, T+1)$ reduces to the non-degenerate interval $(0, 1)$.

The estimation of $I_1^\pm$ is similar to the one of $I_0^\pm$ except when integrating near $k=0$, due to the presence of $\omega$ in the denominator of the relevant integrand. In particular, the only term which needs our attention is the difference 
\eee{
K_1(x, t) :=
\int_{-1}^1
e^{ikx}
\, \frac{e^{i\omega t}-e^{-i\omega t}}{2i\omega} \,  \mathcal{F}\{U_1\}(k)\,dk
=
\int_{-1}^1
e^{ikx}
\, \frac{\sin(\omega t)}{\omega} \,  \mathcal{F}\{U_1\}(k)\,dk.
}
Using the inequality $\left|\frac{\sin(\omega t)}{\omega}\right| \leq t$ and then the Cauchy-Schwarz inequality, we obtain
\aaa{
\no{K_1(x)}_{L^2_t(0,T)}^2  
&=
\int_0^T \left| \int_{-1}^1
e^{ikx}
\, \frac{\sin(\omega t)}{\omega} \,  \mathcal{F}\{U_1\}(k)\,dk 
\right|^2dt  
\nn\\
&\leq
\int_0^T t^2 \left( \int_{-1}^1
  \left|\mathcal{F}\{U_1\}(k)\right| dk 
\right)^2dt  
\nn\\
&\lesssim
T^3 \left( \int_{-1}^1
  \left(1+k^2\right)^{-(s-2)} dk 
\right) \left( \int_{-1}^1
 \left(1+k^2\right)^{s-2}   \left|\mathcal{F}\{U_1\}(k)\right|^2 dk 
\right)
\leq
c_s T^3 \no{U_1}_{H^{s-2}(\R)}^2.
}
The estimation of $\no{\p_t^\ell K_1(x)}_{L_t^2(0, T)}$ for any $\ell \in \N$ can be performed in the same way with \eqref{Const}, since the factor of $\omega$ in the denominator of the integrand of $K_1$ cancels upon differentiation with respect to $t$. 
\end{proof}

\subsection{Forced linear Cauchy problem}

Next, we turn our attention to the forced Cauchy problem with zero initial data
\ddd{\label{flb-ivp}
Z_{tt} - Z_{xx} + Z_{xxxx} &= F(x,t), \quad  x \in \mathbb R, \ t\in (0, T), 
\\
Z(x,0) = Z_t(x,0) & =  0,
}
whose solution is given by
\eee{\label{flb-ivp-sol}
 Z(x,t)= 
 S\big[0, 0; F\big](x, t)
 :=
  \frac{1}{2\pi} \int_{\R} \frac{e^{ikx}}{2i\omega} \left\{e^{i\omega t } \int_0^t e^{-i\omega t'} \mathcal F \{F\}(k,t')\,dt' - e^{-i\omega t} \int_0^t e^{i\omega t'} \mathcal F \{F\}(k,t')\,dt'  \right\}dk. 
 }
The above formula can also be expressed in  Duhamel form as
\eee{ \label{W-sol}
S\big[0, 0; F\big](x, t)
=
\int_0^t  S[0,F(x,t');0](x,t-t')\,dt',
}
where for each $t'\in (0, T)$ the operator $S[0,F(x,t');0](x,t-t')$ is the homogeneous solution operator \eqref{ivp-sol} with $U_0 = 0$, $U_1 = F(x,t')$ and $t$ shifted to $t-t'$. 

\begin{theorem}[Forced Cauchy estimates]\label{fivp-t}
The solution $Z = S\big[0, 0; F\big]$ to the forced linear Cauchy problem~\eqref{flb-ivp} admits the space and time estimates
\aaa{\label{fivp-se-t}
&\no{Z(t)}_{H^s_x(\R)} 
+
\no{\p_t Z(t)}_{H^{s-2}_x(\R)} 
\leq
c(T)\no{F}_{L_t^1((0, T); H_x^{s-2}(\R))}, \quad s\in \R, \ t\in [0, T],
\\
&  \no{\p_x^j Z(x)}_{H_t^{\frac{2s-2j+1}{4}} (0,T)} 
 \leq
 c_j(s, T)
 \begin{cases}
\no{F}_{L_t^1((0,T); H_x^{s-2}(\R))}, 
&  -\frac 12 + j \leq s \leq \frac 32 + j, 
\\
\no{F}_{L_t^2((0,T); H_x^{s-2}(\R))}, 
& \frac 32 + j \leq s \leq \frac 72 + j,  
\end{cases}
\quad j \in \N_0, \ x\in \R.
\label{fivp-te-t}
}
\end{theorem}

\begin{proof}
For the space estimate \eqref{fivp-se-t}, starting from the Duhamel representation \eqref{W-sol}, we combine the triangle inequality with the corresponding homogeneous space estimate \eqref{ivp-se} to readily infer
\aaa{
\no
{Z(t)}_{H^s_x(\R)} 
&\leq
\int_0^t  \no{S[0,F(x,t');0](t-t')}_{H_x^s(\R)} dt'
\leq
2 \sqrt{1+T^2} \no{F}_{L^1_t((0,T);H^{s-2}_x(\R))}, \quad s\in \R, \ t\in [0, T].\label{fivp-se}
}
In addition, using the Leibniz rule while noting that $S[0,F(x,t');0](x,t-t')\big|_{t'=t} = 0$, we have
\eee{\label{zt}
\p_t Z(x,t) 
=
 \int_0^t  \p_t S[0,F(x,t');0](x,t-t')\,dt'
}
so by the triangle inequality and the homogeneous space estimate \eqref{ivp-se-t}, we have
\eee{\label{fivp-se-der}
\no{\p_t Z(t)}_{H_x^{s-2}(\R)}
\leq
 \int_0^t  \no{\p_t S[0,F(x,t');0](x,t-t')}_{H_x^{s-2}(\R)} dt'
\lesssim
\no{F}_{L_t^1((0, T); H_x^{s-2}(\R))}, \quad s\in \R, \ t\in [0, T].
}
Together, \eqref{fivp-se} and \eqref{fivp-se-der} imply the desired space estimate \eqref{fivp-se-t}.

The proof of the time estimate \eqref{fivp-te-t} is more involved. First, as in the case of the homogeneous Cauchy problem (see \eqref{ivp-solx}), differentiating \eqref{flb-ivp-sol} we observe that $\p_x^j Z(x, t) \equiv \p_x^j S\big[0, 0; F\big](x, t) = S\big[0, 0; \p_x^j F\big]$. Thus, it suffices to establish \eqref{fivp-te-t} for $j=0$. Similarly to \eqref{223}, for each $x\in\R$ and each $m \in \N_0$ we employ the physical space Sobolev norm  defined by 
\eee{
\no{Z(x)}_{H_t^m(0,T)} = c_{m, T} \sum_{j=0}^{m} \no{\p_t^j Z(x)}_{L^2_t(0,T)}.
}
For $m=0$, observing that integrand in \eqref{W-sol} makes sense for all $(t-t') \in \R$, we can augment the upper bound of the $t'$-integral from $t$ to $T$ to obtain
\aaa{
\no{Z(x)}^2_{L^2_t(0,T)} 
& \leq
\int_{t=0}^T \left( \int_{t'=0}^t \Big\lvert S[0,F(x,t');0](x,t-t')\Big\rvert\,dt'\right)^2 dt 
\nn\\
 &\leq
 \left[
 \int_{t'=0}^T \Big\|  S[0,F(x,t');0](x,t-t')\Big\|_{L^2_t(0,T)}\,dt'
 \right]^2.
 \label{333}
}
Since $|e^{i\omega t'}|=1$ for $k\in\R$, for each fixed $t'$ the $L^2$ norm with respect to $t$ can be estimated via the homogeneous time estimate \eqref{ivp-te-t} with $U_0 = 0$, $U_1 = F(t')$, $j=0$ and $s= -\frac 12$, giving rise to the estimate
\eee{
\no{Z(x)}_{L^2_t(0,T)} 
 \leq
c(T) \no{F}_{L^1_t((0,T);H_x^{-\frac 52}(\R))}.
}
For $m=1$, recalling \eqref{zt} and proceeding similarly to \eqref{333}, we have
\aaa{
\no{\p_t Z(x)}_{L^2_t(0,T)}
&\leq
\int_{t'=0}^T \no{\p_t S[0,F(x,t');0](x ,t-t')} _{L^2_t (0,T)}dt'
\nn \\
&\leq 
\int_{t'=0}^T \no{S[0,F(x,t');0](x, t-t')} _{H^1_t (0,T)}dt'. 
}
Thus, using once again the homogeneous time estimate \eqref{ivp-te-t}, now for $s=\frac 32$, we obtain
\eee{
\no{\p_t Z(x)}_{L^2_t(0,T)}
\leq c(T) \no{F}_{L^1_t((0,T);H_x^{-\frac12}(\R))}.
}
Moreover, since we have shown that $Z(x)$ is a linear map from  $L^1_t((0, T); H_x^{2m-\frac 52}(\R))$ to $H_t^m(0, T)$ for $m=0,1$, we can interpolate to extend the validity of this result to $0<m<1$. This yields the top half of estimate \eqref{fivp-te-t}. 

Finally, for $m=2$, differentiating \eqref{zt} via the Leibniz rule, we have 
\eee{
\p^2_t Z(x,t) 
=
F(x,t) + \int_0^t  \p^2_t S[0,F(x,t');0](x,t-t')\,dt'.
}
Hence, proceeding as above via Minkowski's integral inequality, we find 
\eee{
\no{\p_t^2 Z(x)}_{L^2_t(0,T)}
\leq
\no{F(x)}_{L_t^2(0, T)}
+
\int_{t'=0}^T \no{S[0,F(x,t');0](x,t-t')}_{H_t^2(0,T)}dt'.
}
The second term can be handled via \eqref{ivp-te-t} with $s=\frac 72$  while, in view of the Sobolev embedding theorem, the first term admits the bound
\eee{
\no{F(x)}_{L_t^2(0, T)}
\leq
\no{F}_{L_t^2((0, T); L_x^\infty(\R))}
\leq
\no{F}_{L_t^2((0, T); H_x^{\frac 32}(\R))}.
}
Therefore, we overall have
\eee{
\no{\p_t^2 Z(x)}_{L^2_t(0,T)}
\leq
\no{F}_{L_t^2((0, T); H_x^{\frac 32}(\R))}
+
c(T)  \no{F}^2_{L^1_t((0,T);H_x^{\frac32}(\R))}.
}
Having shown that $Z(x)$ is a linear map from  $L^2_t((0, T); H_x^{2m-\frac 52}(\R))$ to $H_t^m(0, T)$ for $m=1, 2$, we can interpolate to extend the validity of this result to $1<m<2$. This yields the bottom half of estimate \eqref{fivp-te-t}. 
\end{proof}

In addition to the time estimate \eqref{fivp-te-t} of Theorem \ref{fivp-t}, it will turn out useful to also establish the following time estimate, which is valid for lower values of $s \in (-\frac 32, \frac 12)$ and provides the analogue of \cite[Lemma 11]{h2005}. 
\begin{theorem}[Forced Cauchy time estimate for low $s$]\label{hom-te-t}
Let $j\in \N_0$ and suppose $-\frac 32 + j < s < \frac 12 + j$. Then, the solution $z = S\big[0, 0; \mathfrak F\big]$ to the forced linear Cauchy problem 
\ddd{\label{flb-ivp2}
z_{tt} - z_{xx} + z_{xxxx} &= \mathfrak F(x,t), \quad  x \in \mathbb R, \ t\in \R, 
\\
z(x,0) = z_t(x,0) & =  0,
}
admits the time estimate
\begin{equation}\label{fivp-te-low-t}
\sup_{x\in\mathbb R} \left\| \p_x^j z(x)\right\|_{\dot H_t^{\frac{2s-2j+1}{4}}(\mathbb R)}
\leqslant
c_{s, j} \left\|\mathfrak F\right\|_{L_t^1(\mathbb R; \dot H_x^{s-2}(\mathbb R))}.
\end{equation}
\end{theorem}

\begin{proof}
We first observe that $\p_x^j z(x, t) \equiv \p_x^j S\big[0; \mathfrak F\big] =  S\big[0; \p_x^j \mathfrak F\big]$ thus it suffices to problem \eqref{fivp-te-low-t} for $j=0$. For this task, we employ the technique used in the proof of  \cite[Theorem 2.3]{kpv1993-nls} and \cite[Lemma 11]{h2005}.
In particular, according to the solution formula \eqref{flb-ivp-sol}, 
$Z = Z_+ + Z_-$
where
\eee{\label{Wpm-def}
Z_\pm(x, t) :=  \frac{1}{2\pi} \int_{\R} \frac{e^{ikx\pm i\omega t}}{\pm 2i\omega} \int_0^t e^{\mp i\omega t'} \mathcal F_x\{\mathfrak F\}(k,t')\,dt'  dk. 
}
So, by the triangle inequality, it suffices to establish the estimate for $Z_+$ as the estimation of $Z_-$ is entirely analogous.

We begin by noting that  $Z_+ = I + J$ where
\aaa{
I(x,t)
&:=
\frac{1}{4\pi}\int_{\mathbb R}\textnormal{sgn}(t')
 \int_{\mathbb R} \frac{e^{ikx+i\omega (t-t')}}{2i\omega} \mathcal F_x\{\mathfrak F\}(k,t') dk dt',
   \label{w-I}
\\
J(x, t)
&:= 
\frac{1}{4\pi}\int_{\mathbb R} e^{i\tau t}
\left[\lim_{\varepsilon\rightarrow 0^+} \frac{1}{2\pi} \int_{|\tau-\omega|>\varepsilon} 
e^{ikx} 
\, \frac{\mathcal F\{\mathfrak F\}(k, \tau)}{2i\omega (\tau-\omega)} dk\right] d \tau,
\label{w-J}
}
with $\textnormal{sgn}(\cdot)$ denoting the signum function and  $\mathcal F\{\mathfrak F\}(k,\tau)$ being the Fourier transform of $\mathfrak F(x, t)$ in both $x$ and $t$.  
\\[2mm]
\textbf{Estimation of $I(x,t)$.}
Using duality, we define the  norm of the space $L_x^\infty(\mathbb R; \dot H_t^{\frac{2s+1}{4}}(\mathbb R))$ by
\begin{equation}\label{weak}
\left\|I\right\|_{L_x^\infty(\mathbb R; \dot H_t^{\frac{2s+1}{4}}(\mathbb R))} 
= 
\sup_{\phi} 
\left|
\int_{\mathbb R} \int_{\mathbb R} I(x,t) \phi(x,t) dtdx
\right|
\end{equation}
where $\phi \in C_c^\infty(\mathbb R^2)$ with $\left\|\phi\right\|_{L_x^1(\mathbb R; \dot H_t^{-\frac{2s+1}{4}}(\mathbb R))} = 1$.
By the Cauchy-Schwarz inequality, we have
\begin{align}
\left|\int_{\mathbb R}\int_{\mathbb R}  I(x,t) \phi(x,t) dt dx \right|
&=
\left|
\int_{\mathbb R}
\left(\int_{\mathbb R} \text{sgn}(t')
 e^{-i\omega t'} \frac{\mathcal F_x\{\mathfrak F\}(k,t')}{2i\omega} dt'\right)
\left(\int_{\mathbb R}  e^{i\omega t} \mathcal F_x\{\phi\}(-k, t) dt\right) dk 
\right|
\nn \\
&\leqslant
\left(\int_{\mathbb R}
|k|^{2s}
\left|\int_{\mathbb R} \text{sgn}(t')
 e^{-i\omega t'} \frac{\mathcal F_x\{\mathfrak F\}(k,t')}{2i\omega} dt'\right|^2dk\right)^{\frac 12}
\label{14a}
\\
&\qquad
\cdot \left(
 \int_{\mathbb R}
|k|^{-2s}
\left|\int_{\mathbb R}  e^{i\omega t} \mathcal F_x\{\phi\}(-k, t) dt\right|^2 dk
 \right)^{\frac 12}.\label{14b}
\end{align}
For the term \eqref{14a}, we apply Minkowski's integral inequality and recall the definition of $\omega$ to obtain
\aaa{
\eqref{14a}
&\leqslant
\int_{\mathbb R} 
\left(\int_{\mathbb R}
|k|^{2s}\left|\frac{\mathcal F_x\{\mathfrak F\}(k,t')}{2i\omega}\right|^2dk\right)^{\frac 12}dt' \nn \\
& \lesssim
\int_{\mathbb R} 
\left(\int_{\mathbb R}
|k|^{2(s-2)}\left|\mathcal F_x\{\mathfrak F\}(k,t')\right|^2dk\right)^{\frac 12}dt' 
=
\left\|\mathfrak F\right\|_{L_t^1(\mathbb R; \dot H_x^{s-2}(\mathbb R))}.
\label{14a-est}
}
Regarding the term  \eqref{14b}, we let  
$k = -\xi(\tau)$ for 
$k\leqslant 0$ and $k=\xi(\tau)$ for $k\geqslant 0$, where $\xi$ is given by \eqref{cov}, so that in both cases we have $\tau=-\omega$. Thus, 
\begin{align}
 \hspace{-0.8cm} 
 \eqref{14b}
&=
\Bigg[
 \int_{-\infty}^0
|\xi(\tau)|^{-2s} (-\xi'(\tau))
\left(
\left|\int_{\mathbb R} e^{-i\tau t}\, \mathcal F_x\{\phi\}(\xi(\tau),t) dt\right|^2 
+
\left|\int_{\mathbb R} e^{-i\tau t}\, \mathcal F_x\{\phi\}(-\xi(\tau),t) dt\right|^2 
\right)
d\tau
\Bigg]^{\frac 12}
\nn\\
&\leq
2\left(
\int_{-\infty}^0
|\xi(\tau)|^{-2s} (-\xi'(\tau))
\left(\int_{\R}   \left|\mathcal F_t\{\phi\}(x, \tau)\right| dx \right)^2 d\tau
\right)^{\frac 12}.
\end{align}
Then, using Minkowski's integral inequality and splitting the $\tau$ integral results in
\aaa{
\eqref{14b}
&\leq
2 \int_{\R}
\left(
\int_{-\infty}^{-1}
|\xi(\tau)|^{-2s} (-\xi'(\tau))
 \left|\mathcal F_t\{\phi\}(x, \tau)\right|^2 d\tau
 \right)^{\frac 12}
  dx
 \label{tau1}
 \\
 &\quad
 +
  2 \int_{\R}
\left(
\int_{-1}^0
|\xi(\tau)|^{-2s} (-\xi'(\tau))
 \left|\mathcal F_t\{\phi\}(x, \tau)\right|^2 d\tau
 \right)^{\frac 12}
  dx.
   \label{tau2}
 }
For the first integral, since for $\tau \leq -1$ we have that $\sqrt{-1+\sqrt{1+4\tau^2}} \geq \sqrt{-1+\sqrt 5} > 1$, we can write
$
-\xi'(\tau) =  \sqrt 2|\tau| \left(\sqrt{1+4\tau^2}\sqrt{-1+\sqrt{1+4\tau^2}}\right)^{-1}
\simeq
|\tau|^{-\frac 12}.
$
Therefore, 
\aaa{
\eqref{tau1}
& \lesssim
\int_{\R}
\left(
\int_{-\infty}^{-1}
\left(|\tau|^2\right)^{-\frac{2s+1}{4}}
 \left|\mathcal F_t\{\phi\}(x, \tau)\right|^2 d\tau
 \right)^{\frac 12}
  dx
 \leq
\no{\phi}_{L_x^1(\R; \dot H_t^{-\frac{2s+1}{4}}(\R))} = 1.
\label{14b-est}
}
Concerning \eqref{tau2}, we observe that
$
|\xi(\tau)|^{-2s} (-\xi'(\tau)) 
\simeq
\left(1+\sqrt{1+4\tau^2}\right)^{\frac{2s+1}{2}} |\tau|^{-2s}
\lesssim
|\tau|^{-s-\frac 12}
$
for $\tau \in (-1, 0)$ and $s\leq \frac 12$. Thus, 
\aaa{
\eqref{tau2}
&\lesssim
\int_{\R}
\left(
\int_{-1}^0
|\tau|^{-s-\frac 12}
 \left|\mathcal F_t\{\phi\}(x, \tau)\right|^2 d\tau
 \right)^{\frac 12}
  dx
\leq
\no{\phi}_{L_x^1(\R; \dot H_t^{-\frac{2s+1}{4}}(\R))} = 1.
\label{14c-est}
} 
Overall, combining \eqref{14a-est}, \eqref{14b-est} and \eqref{14c-est} with \eqref{14a}, \eqref{14b} and the definition \eqref{weak}, we conclude that
\begin{equation}
\label{I-est-L1}
\left\|I\right\|_{L_x^\infty(\mathbb R; \dot H_t^{\frac{2s+1}{4}}(\mathbb R))} \lesssim \left\|\mathfrak F\right\|_{L_t^1(\mathbb R; \dot H_x^{s-2}(\mathbb R))},\quad s \leq \frac 12. 
\end{equation}
\noindent
\textbf{Estimation of $J(x,t)$.}
Under the change of variables $k=\xi(\eta)$ for $k\geqslant 0$ and $k = -\xi(\eta)$ for 
$k\leqslant 0$, where $\xi$ is given by \eqref{cov} (note that, in both cases,  $\eta=-\omega \leq 0$), formula \eqref{w-J} becomes
\begin{align}
J(x,t)
&\simeq
\int_{\mathbb R}  e^{i\tau t} 
\left[ -\lim_{\varepsilon \rightarrow 0^+}
\int_{|\tau+\eta|>\varepsilon, \, \eta < 0}
e^{i\xi(\eta)x}\, \frac{\mathcal F\{\mathfrak F\}(\xi(\eta),\tau)}{-\eta(\tau+\eta)}\, \xi'(\eta) d\eta \right] d\tau
\nonumber\\
&\quad
+\int_{\mathbb R}  e^{i\tau t} 
\left[ \lim_{\varepsilon \rightarrow 0^+}
\int_{|\tau+\eta|>\varepsilon, \, \eta < 0}
e^{-i\xi(\eta) x}\, \frac{\mathcal F\{\mathfrak F\}(-\xi(\eta),\tau)}{-\eta(\tau+\eta)}\, (-\xi'(\eta)) d\eta \right] d\tau.
\end{align}
Hence, 
\begin{align}
\left\|J(x)\right\|_{\dot H_t^{\frac{2s+1}{4}}(\mathbb R)}^2
&\lesssim
\int_{\mathbb R}
|\tau|^{s+\frac 12}
\left|
\lim_{\varepsilon \rightarrow 0^+}
\int_{|\tau+\eta|>\varepsilon, \, \eta < 0}
e^{i\xi(\eta)x}   \int_{\mathbb R} e^{-i\tau t} \mathcal F_x\{\mathfrak F\}(\xi(\eta),t) dt \, \frac{ \xi'(\eta)}{\eta(\tau+\eta)} \, d\eta
\right|^2
d\tau
\nonumber\\
&\quad
+
\int_{\mathbb R}
|\tau|^{s+\frac 12}
\left|
\lim_{\varepsilon \rightarrow 0^+}
\int_{|\tau+\eta|>\varepsilon, \, \eta < 0}
e^{-i\xi(\eta) x} \int_{\mathbb R} e^{-i\tau t} \mathcal F_x\{\mathfrak F\}(-\xi(\eta),t) dt \, \frac{\xi'(\eta)}{\eta(\tau+\eta)}\,  d\eta
\right|^2
d\tau
\nn\\
&\leqslant
\int_{\mathbb R}
|\tau|^{s+\frac 12}
\left(
 \int_{\mathbb R}  
\left| 
\lim_{\varepsilon \rightarrow 0^+}
\int_{|\tau-\eta|>\varepsilon, \, \eta < 0}
e^{i\xi(\eta)x}\, \frac{\mathcal F_x\{\mathfrak F\}(\xi(\eta),t)}{\eta(\tau-\eta)}\, \xi'(\eta) d\eta
 \right| dt
\right)^2
d\tau
\nn\\
&\quad
+
\int_{\mathbb R}
|\tau|^{s+\frac 12}
\left(
\int_{\mathbb R} 
\left| 
\lim_{\varepsilon \rightarrow 0^+}
\int_{|\tau-\eta|>\varepsilon, \, \eta < 0}
e^{-i\xi(\eta)x}\, \frac{\mathcal F_x\{\mathfrak F\}(-\xi(\eta),t)}{\eta(\tau-\eta)}\, \xi'(\eta) d\eta
 \right|
  dt
\right)^2
d\tau
\nn\\
&\simeq
\int_{\mathbb R}
|\tau|^{s+\frac 12}
\left(
\int_{\mathbb R} 
\left|
\mathscr H\left\{\varphi_+\right\}(\tau)
\right|
dt
\right)^2
d\tau
+
\int_{\mathbb R}
|\tau|^{s+\frac 12}
\left(
\int_{\mathbb R} 
\left|
\mathscr H\left\{\varphi_-\right\}(\tau)
\right|
dt
\right)^2
d\tau,
\end{align}
where  $\mathscr H\{\cdot\}$ denotes the Hilbert transform and for each  $(x, t) \in \mathbb R^2$ we define
\eee{\label{fpm}
\varphi_\pm (\eta)
=
\left\{
\begin{array}{ll}
e^{i\xi(\eta)x}\, \mathcal F_x\{\mathfrak F\}(\pm \xi(\eta),t)\, \frac{\xi'(\eta)}{\eta}, &\eta < 0, 
\\
0, &\eta>0.
\end{array}
\right.
}

Hence, by Minkowski's integral inequality between the $\tau$ and $t$ integrals, we obtain
\eee{
\left\|J(x)\right\|_{\dot H_t^{\frac{2s+1}{4}}(\mathbb R)}
\lesssim
\int_{\mathbb R} 
\left(
\int_{\mathbb R}
|\tau|^{s+\frac 12}
\left|
\mathscr H\left\{\varphi_+\right\}(\tau)
\right|^2
d\tau
\right)^{\frac 12}
dt
+
\int_{\mathbb R} 
\left(
\int_{\mathbb R}
|\tau|^{s+\frac 12}
\left|
\mathscr H\left\{\varphi_-\right\}(\tau)
\right|^2
d\tau
\right)^{\frac 12}
dt.
\label{J-est-0}
}
According to a classical result from the theory of Calder\' on-Zygmund singular integral operators, the Hilbert transform is bounded from $L^2(\mu)$ to $L^2(\mu)$ when $\mu$ is an $A_2$ weight \cite[Definition 7.1.3 and Theorem 7.4.6]{g2014c}. Using this fact for $\mu = |\tau|^{s+\frac 12}$ with $-\frac 32 < s < \frac 12$, which guarantees that $\mu$ is $A_2$, we obtain  
\eee{
\left\|J(x)\right\|_{\dot H_t^{\frac{2s+1}{4}}(\mathbb R)}
\lesssim
\int_{\mathbb R} \left(\int_{\mathbb R} |\eta|^{s+\frac 12}  \left|\varphi_+(\eta)\right|^2 d\eta \right)^{\frac 12} dt
+
\int_{\mathbb R} \left(\int_{\mathbb R} |\eta|^{s+\frac 12}  \left|\varphi_-(\eta)\right|^2 d\eta \right)^{\frac 12} dt.
}
Substituting for $\varphi_\pm$ via \eqref{fpm} and switching variables back to $k$ by letting $k = \xi(\eta)$ and $k = -\xi(\eta)$, respectively, we find, after also taking into account that $s - \frac 32 \leq 0$, 
\begin{align}
\left\|J(x)\right\|_{\dot H_t^{\frac{2s+1}{4}}(\mathbb R)}
&\lesssim
\int_{\mathbb R} \left(\int_0^\infty \left|k\left(1+k^2\right)^{\frac 12}\right|^{s-\frac 32}  \left| \mathcal F_x\{\mathfrak F\}(k,t)\right|^2  \frac{\left(1+k^2\right)^{\frac12}}{(1+2k^2)} dk \right)^{\frac 12} dt
\nn\\
&\quad
+
\int_{\mathbb R} \left(\int_{-\infty}^0 \left|k\left(1+k^2\right)^{\frac 12}\right|^{s-\frac 32}  \left| \mathcal F_x\{\mathfrak F\}(k,t)\right|^2  \frac{\left(1+k^2\right)^{\frac12}}{(1+2k^2)}  dk \right)^{\frac 12} dt
\nn\\
&\leq
\int_{\mathbb R} \left(\int_0^\infty |k^2|^{s-\frac 32}  \left| \mathcal F_x\{\mathfrak F\}(k,t)\right|^2  \frac{1}{|k|} dk \right)^{\frac 12} dt 
+ \int_{\mathbb R} \left(\int_{-\infty}^0 |k^2|^{s-\frac 32}  \left| \mathcal F_x\{\mathfrak F\}(k,t)\right|^2  \frac{1}{|k|}  dk \right)^{\frac 12} dt 
\nn\\
&\leq
2\left\|\mathfrak F\right\|^2_{L_t^1(\mathbb R; \dot H_x^{s-2}(\mathbb R))},
\quad -\frac 32 < s < \frac 12, \ x\in \mathbb R.
\label{J-est-L1}
\end{align}

Combining estimate \eqref{I-est-L1} for $I$ and estimate \eqref{J-est-L1} for $J$  with the writing $Z_+ = I + J$, we deduce the desired time estimate \eqref{fivp-te-low-t} for $Z_+$. The term $Z_-$ can be estimated similarly, thus the proof of Theorem \ref{hom-te-t} is complete.
\end{proof}

The estimates \eqref{fivp-te-t} and \eqref{fivp-te-low-t} imply that the quantity $(\p_x +\gamma I)Z(x, t)$   belongs to the space $H_t^{\frac{2s-1}4}(0,T)$. Specifically,  if $s \leq \frac 12$, then 
\eee{
\no{(\p_x +\gamma I)Z(x)}_{H_t^{\frac{2s-1}4}(0,T)}
\leq
\no{\p_x Z(x)}_{\dot H_t^{\frac{2s-1}4}(0,T)}
+
|\gamma| \no{Z(x)}_{H_t^{\frac{2s+1}4}(0,T)}
}
so, using \eqref{fivp-te-low-t} with $j=1$ and $\mathfrak F(x, t) = \chi_{[0, T]}(t) F(x, t)$ for the norm of $\p_x Z$ and also \eqref{fivp-te-t} with $j=0$ for the norm of $Z$, we find
\eee{\label{rte1}
\no{(\p_x +\gamma I)Z(x)}_{H_t^{\frac{2s-1}4}(0,T)}
\leq
c_s 
\no{F}_{L_t^1((0, T); \dot H_x^{s-2}(\R))},
\quad
-\frac 12 < s \leq \frac 12, \ x\in \R.
}
If $\frac 12 \leq s \leq  \frac 52$, then we only need \eqref{fivp-te-t}, once with  $j=1$ (for $\p_x Z$) and once with $j=0$ and $s-1$ in place of $s$ (for $Z$). We then obtain
\eee{\label{rte2}
\no{(\p_x +\gamma I)Z(x)}_{H_t^{\frac{2s-1}4}(0,T)}
\leq
c(s, T)
\no{F}_{L_t^1((0, T); H_x^{s-2}(\R))},
\quad
\frac 12 \leq s \leq \frac 52, \ x\in \R.
}
Finally, if $\frac 52 \leq s \leq  \frac 92$, then employing \eqref{fivp-te-t} as for the previous range of $s$ yields
\eee{\label{rte3}
\no{(\p_x +\gamma I)Z(x)}_{H_t^{\frac{2s-1}4}(0,T)}
\leq
c(s, T)
\no{F}_{L_t^2((0, T); H_x^{s-2}(\R))},
\quad
\frac 52 \leq s \leq \frac 92,\ x\in \R.
}
Similarly,  for each $x \in \R$, the term $(\p_x^2 +\delta \p_x )Z(x, t)$ belongs to the space $H_t^{\frac{2s-3}4}(0,T)$ with the estimate
\eee{\label{fivp-te5}
\no{(\p_x^2 +\delta \p_x )Z(x)}_{H_t^{\frac{2s-3}4}(0,T)}
 \leq
 c(s, T)
\begin{cases}
\no{F}_{L_t^1((0, T); \dot H_x^{s-2}(\R))}, & \frac12 < s \leq \frac32, \\
\no{F}_{L_t^1((0, T); H_x^{s-2}(\R))}, & \frac32 \leq s \leq \frac72, \\
\no{F}_{L_t^2((0, T); H_x^{s-2}(\R))}, & \frac72 \leq s \leq \frac{11}2,
\end{cases}
\quad
x\in\R.
}

\section{Linear and nonlinear theory for $s>\frac 12$}
\label{l-nl-high-s}

In this section, we establish the central linear estimate \eqref{lin-est} of Theorem \ref{linear-t} and then combine it with a contraction mapping argument in order to prove Theorem \ref{lwp-t} for the local well-posedness of the nonlinear problem~\eqref{gbous} in the case of high regularity.

\subsection{Proof of Theorem \ref{linear-t}}

The results of Section \ref{cauchy-s} allow us to perform a decomposition of the forced linear problem \eqref{lgbous} into simpler component problems that will be estimated individually.
For any $s\in \R$, let $U_0, U_1$ be extensions of the initial data $u_0, u_1$ from $H^s(0, \infty)$ to $H^s(\R)$ and from $H^{s-2}(0, \infty)$ to $H^{s-2}(\R)$, respectively, such that  
\eee{\label{ic-ext}
\no{U_0}_{H^s(\R)} \leq 2 \no{u_0}_{H^s(0,\infty)},
\quad
\no{U_1}_{H^{s-2}(\R)} \leq 2 \no{u_1}_{H^{s-2}(0,\infty)}.
}
The existence of such extensions is justified by viewing $H^s(0, \infty)$ as the restriction of $H^s(\R)$ on the half-line through the infimum norm
\eee{
\no{u}_{H^s(0, \infty)} := \inf\left\{\no{U}_{H^s(\R)}: U\big|_{(0, \infty)} = u\right\}.
}
Similarly, for each $t\in (0, T)$, let $W(t)\in H_x^s(\R)$ be an extension of $w(t)\in H_x^s(0, \infty)$ such that 
\eee{\label{f-ext}
\no{W(t)}_{H_x^s(\R)} \leq 2 \no{w(t)}_{H_x^s(0, \infty)}, \quad t\in (0, T). 
}
Importantly, note that for each $t\in (0, T)$ we have $\p_x^2 W\big|_{(0, \infty)} = \p_x^2 w$ (in the weak sense).  Then, let $W_0(t)$ denote the zero extension of $W(t)$ from $(0, T)$ to $(0, T+1)$, so that 
\eee{\label{f-ext2}
\no{W_0}_{L_t^p((0, T+1); H_x^s(\R))} = \no{W}_{L_t^p((0, T); H_x^s(\R))}
\leq
2 \no{w}_{L_t^p((0, T); H_x^s(0, \infty))}, \quad 1 \leq p \leq \infty, \ s\in\R, \ T>0.
}

Consider now the homogeneous Cauchy problem \eqref{ivp} on $\R \times (0, T)$ with initial data given by the extensions $U_0, U_1$ specified above and corresponding solution denoted by $U = S\big[U_0, U_1; 0\big]$. 
Furthermore, consider the forced Cauchy problem \eqref{flb-ivp} on $\R \times (0, T+1)$ (i.e. with $T$ in \eqref{flb-ivp} replaced by $T+1$) in the case of the forcing $F=-\p_x^2 W_0$ and with corresponding solution denoted by $Z = S\big[0, 0; -\p_x^2 W_0\big]$. Then, the solution of problem~\eqref{lgbous} can be written as
\eee{\label{dec}
S\big[u_0, u_1, \alpha, \beta; -\p_x^2 w\big]
=
S\big[U_0, U_1; 0\big]\big|_{x>0}
+
S\big[0, 0; -\p_x^2 W_0\big]\big|_{x>0, \, t\in (0, T)}
+
S\big[0, 0, \varphi_0, \psi_0; 0\big]
}
with the third component on the right side satisfying the following half-line problem with zero initial data and zero forcing
\ddd{\label{ribvp-0}
&u_{tt} - u_{xx} + u_{xxxx} = 0, \quad  x \in (0, \infty), \ t\in (0, T), 
\\
&u(x,0) = 0, 
\quad
u_t(x,0) = 0, 
\\
&u_x(0,t)+\gamma u(0,t) =   \varphi_0(t), 
\quad
u_{xx}(0,t)+\delta u_{x}(0,t) =   \psi_0(t),
}
where, for $U = S\big[U_0, U_1; 0\big]$ and $Z = S\big[0, 0; -\p_x^2 W_0\big]$ as above, we define
\ddd{
\varphi_0(t) &:= \alpha(t) - \left(\p_x + \gamma I\right)U(0, t) - \left(\p_x + \gamma I\right)Z(0, t)\big|_{t\in (0, T)}, 
\\
 \psi_0(t) &:= \beta(t) - \left(\p_x^2 + \delta \p_x\right)U(0, t) - \left(\p_x^2 + \delta \p_x\right)Z(0, t)\big|_{t\in (0, T)}.
}

The problem \eqref{ribvp-0} is almost the reduced problem \eqref{lgbous-r}, the only difference being the compact support condition present in \eqref{lgbous-r}. In particular, the regularity of the boundary data $\varphi_0$ and $\psi_0$ is the same with that of the data in~\eqref{lgbous-r}. Indeed, using the homogeneous Cauchy time estimate \eqref{ivp-te-t} and the forced Cauchy time estimates \eqref{rte1}-\eqref{fivp-te5} but with $T$ replaced by $T+1$, along with the fact that $\no{\p_x^2 W_0}_{\dot H_x^{s-2}(\R)} = \no{W_0}_{\dot H_x^{s}(\R)}$  and $\no{\p_x^2 W_0}_{H_x^{s-2}(\R)} \leq \no{W_0}_{H_x^s(\R)}$, we have
\aaa{
\no{\varphi_0}_{H^{\frac{2s-1}{4}}(0, T)}
&\leq
\no{\alpha}_{H^{\frac{2s-1}{4}}(0, T)} 
+
c(s, T) \left(\no{U_0}_{H^s(\R)} + \no{U_1}_{H^{s-2}(\R)}\right)
\nn\\
&\quad
+
c(s, T+1) \begin{cases} \no{W_0}_{L_t^1((0, T+1); \dot H_x^{s}(\R))}, &-\frac 12<s<\frac 12
\\
\no{W_0}_{L_t^1((0, T+1); H_x^{s}(\R))}, &\frac 12\leq s\leq\frac 52
\\
\no{W_0}_{L_t^2((0, T+1); H_x^{s}(\R))}, &\frac 52\leq s\leq\frac 92
\end{cases}.
}
Hence, by the extension inequalities \eqref{ic-ext} and \eqref{f-ext2}, 
\aaa{
\no{\varphi_0}_{H^{\frac{2s-1}{4}}(0, T)}
&\leq 
\no{\alpha}_{H^{\frac{2s-1}{4}}(0, T)} 
+
c(s, T) \left(\no{u_0}_{H^s(0, \infty)} + \no{u_1}_{H^{s-2}(0, \infty)}\right)
\nn\\
&\quad
+
c(s, T+1) \begin{cases} \no{w}_{L_t^1((0, T); \dot H_x^{s}(0, \infty))}, &-\frac 12<s<\frac 12
\\
\no{w}_{L_t^1((0, T); H_x^{s}(0, \infty))}, &\frac 12\leq s\leq\frac 52
\\
\no{w}_{L_t^2((0, T); H_x^{s}(0, \infty))}, &\frac 52\leq s\leq\frac 92
\end{cases},
\label{phi0-est}
}
which shows that $\varphi_0 \in H^{\frac{2s-1}{4}}(0, T)$. Similarly, $\psi_0 \in H^{\frac{2s-3}{4}}(0, T)$ with
\aaa{
\no{\psi_0}_{H^{\frac{2s-3}{4}}(0, T)}
&\leq
\no{\beta}_{H^{\frac{2s-3}{4}}(0, T)} 
+
c(s, T) \left(\no{u_0}_{H^s(0, \infty)} + \no{u_1}_{H^{s-2}(0, \infty)}\right)
\nn\\
&\quad
+
c(s, T+1)
\begin{cases}
\no{w}_{L_t^1((0,T);\dot H_x^{s}(0,\infty))}, & \frac12 < s< \frac32 \\
\no{w}_{L_t^1((0,T);H_x^{s}(0,\infty))}, & \frac32 \leq s \leq \frac72 \\
\no{w}_{L_t^2((0,T);H_x^{s}(0,\infty))}, & \frac72 \leq s \leq \frac{11}2
\end{cases}.
\label{psi0-est}
}
Importantly, note that the constant $c(s, T+1)$ in the above estimates remains bounded as $T\to 0^+$ due to the fact that it originates from the time estimate \eqref{fivp-te-t} but with $T+1$ in place ot $T$, in which case the relevant interval $(0, T+1)$ reduces to the non-degenerate interval $(0, 1)$ as $T \to 0^+$.

Next, we construct suitable extensions of $\varphi_0, \psi_0$ that have compact support (this will allow us to employ Theorem~\ref{r-t} for the third component in \eqref{dec}). 
First, by the infimum approximation property, let $\Phi \in H^{\frac{2s-1}{4}}(\R)$ be an extension of $\varphi_0$ such that
\eee{\label{phi-ineq}
\no{\Phi}_{H^{\frac{2s-1}{4}}(\R)}
\leq
2 \no{\varphi_0}_{H^{\frac{2s-1}{4}}(0, T)},
\quad
s\in\R.
}
If  $\frac 12 \leq s < \frac 32$, then by \cite[Theorem 11.4]{lm1972} the function $\chi_{[0, T+1]} \Phi$ (i.e. the zero extension of $\Phi|_{(0, T+1)}$) belongs to the Sobolev space $H^{\frac{2s-1}{4}}(\R)$ with the inequality
\eee{\label{zero-ext}
\no{\chi_{[0, T+1]} \Phi}_{H^{\frac{2s-1}{4}}(\R)}
\leq
c_{s, T+1} \no{\Phi}_{H^{\frac{2s-1}{4}}(0, T+1)}
\leq
2 c_{s, T+1} \no{\varphi_0}_{H^{\frac{2s-1}{4}}(0, T)}
}
with the last inequality due to \eqref{phi-ineq}. Note that in the limit $T\to 0^+$ the interval $(0, T+1)$ tends to $(0, 1)$, which is a non-degenerate interval and hence our constant $c_{s, T+1}$ remains finite as $T\to 0^+$.

If $\frac 32 < s \leq \frac 92$, then more care is required due to the existence of Sobolev traces of $\varphi_0$ at $t=0, T$. Specifically, thanks to our space and time estimates \eqref{ivp-se-t}, \eqref{ivp-te-t}, \eqref{fivp-se-t}, \eqref{fivp-te-t}, for $s> \frac 32$ we have continuity of $U, U_x, Z, Z_x$ in both $x$ and $t$. Thus, since  $U_0$ is an extension of $u_0$ and $U_0, U_0'$ are continuous functions of $x$,  we deduce
\eee{\label{phi0}
\varphi_0(0) =  \alpha(0) -  \left(\p_x + \gamma I\right)U(x, 0)\big|_{x=0} - \left(\p_x + \gamma I\right)Z(x, 0)\big|_{x=0}
=
\alpha(0) -  \left(u_0'(0) + \gamma u_0(0)\right) = 0, \quad s>\frac 32,
}
after using the first compatibility condition in \eqref{comp} for the last step. Similarly,
\eee{
\varphi_0'(0) 
= 
 \alpha'(0) -  \left(\p_x + \gamma I\right)U_t(x, 0)\big|_{x=0} - \left(\p_x + \gamma I\right)Z_t(x, 0)\big|_{x=0}
=
\alpha'(0) -  \left(u_1'(0) + \gamma u_1(0)\right)= 0, \quad s>\frac 72.
\label{phi0p}
}
Let $\theta \in C_c^\infty(\R)$ such that $0 \leq \theta(t) \leq 1$ for all $t\in \R$, $\theta(t) \equiv 1$ on $[- T, T]$ and $\theta(t) \equiv 0$ on $[-T-\frac 12, T+\frac 12]^c$. Then, since $\frac{2s-1}{4}>\frac 12$, by the algebra property the globally defined function
$\Phi_\theta(t) := \theta(t) \Phi(t)$ (with $\Phi$ as above)
satisfies 
\eee{\label{noname}
\no{\Phi_\theta}_{H^{\frac{2s-1}{4}}(\R)}
\leq
c_s \no{\theta}_{H^{\frac{2s-1}{4}}(\R)} \no{\Phi}_{H^{\frac{2s-1}{4}}(\R)}
\leq
c(s, T) \no{\varphi_0}_{H^{\frac{2s-1}{4}}(0, T)}
}
where the constant $c(s, T)$ remains finite at $T\to 0^+$ due to the fact that $\theta$ is smooth and its support becomes  $[-\frac 12, \frac 12]$ in that limit. In particular, \eqref{noname} shows that $\Phi_\theta|_{(0, T+1)} \in H^{\frac{2s-1}{4}}(0, T+1)$. 
In addition, by continuity in $t$ (since $s>\frac 32$) and the condition \eqref{phi0}, we have $\Phi_\theta(0) = \varphi_0(0) = 0$ and, moreover, $\Phi_\theta(T+1) = 0$ since $\theta(T+1) = 0$. 
Hence, for $\frac 32 < s \leq \frac 72$, Theorem 11.5 of \cite{lm1972} implies that $\Phi_\theta|_{(0, T+1)} \in H_0^{\frac{2s-1}{4}}(0, T+1)$. In turn, for $\frac 32 < s < \frac 72$, Theorem 11.4 of \cite{lm1972} allows us to conclude that $\chi_{[0, T+1]} \Phi_\theta \in H^{\frac{2s-1}{4}}(\R)$ with
\eee{\label{r0r}
\no{\chi_{[0, T+1]} \Phi_\theta}_{H^{\frac{2s-1}{4}}(\R)}
\leq
c_{s, T+1}  \no{\Phi_\theta}_{H^{\frac{2s-1}{4}}(0, T+1)}
\leq
c_{s, T+1} \, c(s, T)  \no{\varphi_0}_{H^{\frac{2s-1}{4}}(0, T)}
}
with the last inequality thanks to \eqref{noname}.

For the remaining range of $\frac 72 < s \leq \frac{9}{2}$ in \eqref{phi0-est}, using \eqref{phi0}, \eqref{phi0p}, the definition of $\theta$ and the fact that the derivative in $t$ is now continuous, we have
$\Phi_\theta'(0) = \theta'(0) \Phi(0) + \theta(0) \Phi'(0)
= 0
$
and
$\Phi_\theta'(T+1) = 0$.
These additional vanishing traces allow us to employ once again Theorem 11.5 of \cite{lm1972} to infer that $\Phi_\theta|_{(0, T+1)} \in H_0^{\frac{2s-1}{4}}(0, T+1)$ and then Theorem 11.4 in \cite{lm1972} to deduce the inequality \eqref{r0r} also for $\frac 72 < s \leq \frac{9}{2}$. 
Overall, the function
\eee{
\rho_0(t) := \left\{\begin{array}{ll} \chi_{[0, T+1]}(t) \, \Phi(t), &\frac 12 \leq s < \frac 32, \\[1mm]  \chi_{[0, T+1]}(t) \, \Phi_\theta(t), &\frac 32 < s \leq \frac 92, \ s \neq \frac 72, \end{array}\right.
}
is an extension of $\varphi_0$ from $(0, T)$ to $\R$ such that $\text{supp}(\rho_0) \subset (0, T+1)$ and
\eee{\label{r0r2}
\no{\rho_0}_{H^{\frac{2s-1}{4}}(\R)}
\leq
c_{s, T}  \no{\varphi_0}_{H^{\frac{2s-1}{4}}(0, T)}
}
where $c_{s, T}$ is a constant that remains bounded as $T\to 0^+$. 

For $\frac 32<s\leq \frac 92$ with $s\neq \frac 52$, the above argument for $\varphi_0$ can be readily adapted to also construct the desired extension $\zeta_0 \in H^{\frac{2s-3}{4}}(\R)$ for $\psi_0 \in H^{\frac{2s-3}{4}}(0, T)$ such that $\text{supp}(\zeta_0) \subset (0, T+1)$ and 
\eee{\label{z0z}
\no{\zeta_0}_{H^{\frac{2s-3}{4}}(\R)} \leq c_{s, T} \no{\psi_0}_{H^{\frac{2s-3}{4}}(0, T)}.
} 
In fact, \eqref{z0z} is also valid for $\frac 12 < s < \frac 32$ (which corresponds to $-\frac 12 < \frac{2s-3}{4} < 0$) thanks to the following lemma, which follows from \cite[Theorem~11.4]{lm1972} and a duality argument.
\begin{lemma}[\cite{mmj2026}]
\label{Chris-l}
Let $-\frac 12 < s <0$ and $-\infty \leq a < b \leq \infty$. If $f \in H^s(a, b)$, then the zero extension $F_0$ of $f$ belongs to $H^s(\R)$ and satisfies the estimate
$\no{F_0}_{H^s (\R)} \leq c_{s, a, b} \no{f}_{H^s (a, b)}$.
\end{lemma}

With $\rho_0$ and $\zeta_0$ defined as above, we observe that problem \eqref{ribvp-0} can be embedded in the problem 
\ddd{\label{ribvp-01}
&u_{tt} - u_{xx} + u_{xxxx} = 0, \quad  x \in (0, \infty), \ t\in (0, T+1), 
\\
&u(x,0) = 0, 
\quad
u_t(x,0) = 0, 
\\
&u_x(0,t)+\gamma u(0,t) = \rho_0(t), \quad u_{xx}(0,t)+\delta u_{x}(0,t) =  \zeta_0(t)
}
Since $\text{supp}(\rho_0), \, \text{supp}(\zeta_0) \subset (0, T+1)$, we can estimate \eqref{ribvp-01} directly via Theorem \ref{r-t}  with $T$ replaced by $T+1$ and $(\varphi, \psi) = (\rho_0, \zeta_0)$ to infer
\eee{\label{r-se2}
\no{S\big[0, 0, \rho_0, \zeta_0; 0\big]}_{X_T^s}
 \leq  \left[c(s)+c(s, T+1) \sqrt{T+2} \big(1-e^{-(T+2)}\big)^{-1} \right] \left(\no{\rho_0}_{H^{\frac{2s-1}{4}}(\R)} + \no{\zeta_0}_{H^{\frac{2s-3}{4}}(\R)}\right).
}
Finally, using our extension inequalities \eqref{zero-ext}, \eqref{r0r2} and \eqref{z0z} together with the fact that $S\big[0, 0, \rho_0, \zeta_0; 0\big]\big|_{(0, T)} = S\big[0, 0, \varphi_0, \psi_0; 0\big]$, we conclude that the solution to problem \eqref{ribvp-0} satisfies
\ddd{
\no{S\big[0, 0, \varphi_0, \psi_0; 0\big]}_{X_T^s}
&\leq
c_{s, T} \left[c(s)+c(s, T+1) \sqrt{T+2} \big(1-e^{-(T+2)}\big)^{-1} \right] 
\\
&\quad
\cdot 
\left(\no{\varphi_0}_{H^{\frac{2s-1}{4}}(0, T)} + \no{\psi_0}_{H^{\frac{2s-3}{4}}(0, T)}\right),
\quad
\frac 12 < s\leq \frac 92, \ s\neq \frac 32, \frac 72. \label{r-se3}
}

Overall, for $\frac 12 < s\leq \frac 92$ with $s\neq \frac 32, \frac 72$, the decomposition \eqref{dec}, the estimates \eqref{ivp-se-t}, \eqref{fivp-se-t}, \eqref{r-se3}, \eqref{phi0-est}, \eqref{psi0-est}, and the extension inequalities~\eqref{ic-ext},~\eqref{f-ext} imply the desired estimate \eqref{lin-est} for the forced linear ``good'' Boussinesq problem \eqref{lgbous}, completing the proof of Theorem~\ref{linear-t}.

\subsection{Proof of Theorem \ref{lwp-t}}\label{lwp-s}

\noindent
\textit{Existence.} Recalling the definition \eqref{xst} of  $X_T^s$ and introducing the notation $S(u) := S\big[u_0, u_1, \alpha, \beta; -\p_x^2(u^2)\big]$, we combine estimate \eqref{lin-est} for $w=u^2$ with the algebra property in $H^s$ to infer
\aaa{\label{nonlin-est}
\no{S(u)}_{X_T^s}
&\leq
c(s, T) 
\Big(
\no{u_0}_{H^{s}(0, \infty)}
+
\no{u_1}_{H^{s-2}(0, \infty)}
+
\no{\alpha}_{H^{\frac{2s-1}{4}}(0, T)} 
+
\no{\beta}_{H^{\frac{2s-3}{4}}(0, T)} 
\nn\\
&\hspace*{1.8cm}
+
\sqrt T \no{u}_{X_T^s}^2\Big),
\quad
\frac12 < s < \frac92, \ s\neq \frac 32, \frac 52, \frac72.
}
Suppose $B(0,r) \subseteq X_T^s$ is a ball centered at $0$ with radius 
\eee{
r:= 2  c(s,T) \Big( \no{u_0}_{H^{s}(0, \infty)}
+
\no{u_1}_{H^{s-2}(0, \infty)}
+
\no{\alpha}_{H^{\frac{2s-1}{4}}(0, T)} 
+
\no{\beta}_{H^{\frac{2s-3}{4}}(0, T)}  \Big).
} 
Then, if $u \in \overline{B(0,r)}$, estimate \eqref{nonlin-est} implies
\eee{
	\no{S(u)}_{X_T^s} \leq \frac r2 +  c(s,T) \sqrt T  \, r^2,  
}
hence a sufficient condition for $S(u) \in \overline{B(0,r)}$ is $2r c(s,T) \sqrt T \leq 1$.  Further, for any $u, v \in \overline{B(0,r)}$, by \eqref{lin-est} and the algebra property in $H^s$, we have  
\eee{
\no{S(u) - S(v)}_{X_T^s}
\equiv
\no{S[0, 0, 0, 0; -\p_x^2(u^2-v^2)]}_{X_T^s} 
\leq  c(s,T) \sqrt T \no{u^2-v^2}_{X_T^s} 
\leq 
2r c(s,T)\sqrt T  \no{u-v}_{X_T^s}, \label{fp-est}
}
so the condition $2r c(s,T)\sqrt T < 1$, which corresponds to \eqref{contr-cond} with $c(s, T)$ replaced by $2c(s, T)$, guarantees that the map $u \mapsto S(u)$ is a contraction on $\overline{B(0, r)}$.
Then, by Banach's fixed point theorem, for $T$ satisfying \eqref{contr-cond} we have a unique fixed point of $u \mapsto S(u)$ in $\overline{B(0, r)} \subset X_T^s$, which amounts to a unique solution (in the sense of the integral equation $u=S(u)$) to the nonlinear problem \eqref{gbous} in $\overline{B(0, r)} \subset X_T^s$.
\\[2mm]
\noindent
\textit{Extending uniqueness to the whole of $X_T^s$.} 
We adapt the approach given in \cite[Proposition 4.2]{cw1990} to the framework of initial-boundary value problems for equations with two temporal derivatives.
Let $w=S(w)$ and $v=S(v)$ be two solutions to problem \eqref{gbous}  such that $w,v \in X_T^s$. Suppose, to the contrary, that there exists $t\in[0,T]$ at which $w(t) \neq v(t)$, and let 
\eee{\label{t-inf}
	t_{\inf} := \inf A \geq 0, \quad A := \{ t \in [0,T] : \text{ either } w(t) \neq v(t) \text{ or } w_t(t) \neq v_t(t) \}.
}

First, suppose $t_{\inf}>0$. Take $0 \leq t_n < t_{\inf}$ such that $t_n \to t_{\inf}^-$ as $n \to \infty$. Then,  $w(t_n)=v(t_n)$ \textit{and} $w_t(t_n)=v_t(t_n)$ for all $t_n$. Also, since $w,v \in X_T^s$,  $w, v$ are continuously differentiable on $[0,T]$. Thus, 
\eee{
w(t_{\inf}) = v(t_{\inf})=: \sigma_0 \in H^s(0,\infty),
\quad
w_t(t_{\inf}) = v_t(t_{\inf})=: \sigma_1 \in H^{s-2}(0,\infty).
}
which shows that $t_{\inf} \notin A$ and $t_{\inf}<T$ (since $A \neq \emptyset$).

Now, set $z_1(t) = w(t+t_{\inf})$ and  $z_2(t) = v(t+t_{\inf})$ and note that $z_1$ and $z_2$ are both solutions (as fixed points of the relevant iteration map) to the ``good'' Boussinesq equation on $(0, \infty) \times (0, T-t_{\inf})$ that satisfy the same initial and boundary conditions:
\ddd{\label{shifted-ibc}
& z_1(0)=z_2(0)= \sigma_0, \qquad \p_t z_1 (0) = \p_t z_2 (0)= \sigma_1, \\
& (\p_x+\gamma I) z_1 \big\vert_{x=0} = \alpha (\cdot + t_{\inf}) =: \alpha_{\inf}, 
\qquad
(\p_x^2+\delta \p_x) z_2 \big\vert_{x=0} = \beta (\cdot + t_{\inf}) =: \beta_{\inf}.
}
If $t_{\inf} =0$, then we set $z_1=w$ and $z_2=v$ and modify the above boundary conditions accordingly. 
Thus, from the contraction inequality \eqref{fp-est}, for all $T_* \in (0, T-t_{\inf}]$ we have
\eee{
	\no{z_1-z_2}_{X_{T_*}^s} 
	\leq
	 c_{\inf}(s,T_*) \sqrt{T_*} 
	 \left( \no {z_1}_{X_{T_*}^s}+\no {z_1}_{X_{T_*}^s}\right)
	 \no{z_1-z_2}_{X_{T_*}^s},
}
where $c_{\inf}$ remains bounded as $T_* \to 0^+$. Therefore, choosing $T_*>0$ sufficiently small, we have
\eee{
	\no{z_1-z_2}_{X_{T_*}^s}  \leq \frac12 \no{z_1-z_2}_{X_{T_*}^s} \ \Rightarrow \  \no{z_1-z_2}_{X_{T_*}^s} =0.
}
By the definition \eqref{xst} of $X_{T^*}^s$, this implies $w=v$ \textit{and} $w_t = v_t$ for all $t \in [t_{\inf}, t_{\inf}+T_*]$ and so $t_{\inf}+T_*$ is a lower bound of the set $A$ which is greater than $t_{\inf}$, leading to a contradiction. Thus, we must have $w=v$ in $X_T^s$ and our solution obtained by fixed point iteration is unique in the entire solution space. 
\\[2mm]
\noindent
\textit{Continuous dependence on the data.} We will show that the data-to-solution map is continuous and, more precisely, locally Lipschitz. 
Let $T>0$ and $d\in \mathcal D_T := H^s(0,\infty) \times H^{s-2}(0,\infty) \times H^{\frac{2s-1}{4}}(0,T) \times H^{\frac{2s-3}{4}}(0,T)$ be such that the contraction condition \eqref{contr-cond} is satisfied, namely
\eee{\label{contr-cond3}
4c(s,T)^2 \sqrt T < \frac{1}{\no{d}_{\mathcal D_T}},
}
where $\no{d}_{\mathcal D_T}$ denotes the sum of the individual data norms. Thus, by our existence theory, $d \in \mathcal D_T$ yields a (unique) solution in $\overline{B(0, r(T))} \subset X_T^s$, where $r(T) := 2c(s, T) \no{d}_{\mathcal D_T}$. 
Let $R=R(d, T)>0$ be small enough so that
\eee{\label{contr-cond2}
4c(s,T)^2 \sqrt T < \frac{1}{\no{d}_{\mathcal D_T}+R}
}
and denote by $B(d,R)$ the  open ball in $\mathcal D_T$ centered at $d$ and with radius $R$. For any two data points $d_u = (u_0,u_1, \alpha_u,\beta_u)$ and $d_v= (v_0,v_1, \alpha_v,\beta_v)$ in $B(d,R) \subset \mathcal D_T$, we have
\ddd{\label{duv}
	\no{d_u}_{\mathcal D_T} 
	& \equiv \no{u_0}_{H^s(0,\infty)} + \no{u_1}_{H^{s-2}(0,\infty)} + 
	\no{\alpha_u}_{H^{\frac{2s-1}{4}}(0,T) } + \no{\beta_u}_{H^{\frac{2s-3}{4}}(0,T)} 
	 < \no{d}_{\mathcal D_T}+ R, 
	 \\
	\no{d_v}_{\mathcal D_T} &\equiv \no{v_0}_{H^s(0,\infty)} + \no{v_1}_{H^{s-2}(0,\infty)} + 
	\no{\alpha_v}_{H^{\frac{2s-1}{4}}(0,T) } + \no{\beta_v}_{H^{\frac{2s-3}{4}}(0,T)} < \no{d}_{\mathcal D_T}+ R.
}
Thus, from \eqref{contr-cond2} we obtain the conditions
\eee{\label{contr-condj}
4c(s,T)^2 \sqrt T < \frac{1}{\no{d_u}_{\mathcal D_T}}, \quad 
4c(s,T)^2 \sqrt T < \frac{1}{\no{d_v}_{\mathcal D_T}},
}
which from our existence theory (specifically, the contraction condition \eqref{contr-cond3}) guarantee corresponding unique solutions $u \in \overline{B(0, r_u)} \subset X_T^s$ and $v \in \overline{B(0, r_v)} \subset X_T^s$, where $r_u(T) := 2c(s, T) \no{d_u}_{\mathcal D_T}$ and $r_v(T) := 2c(s, T) \no{d_v}_{\mathcal D_T}$. 
In fact, $u$ and $v$ arise as fixed points of the relevant iteration maps 
in the corresponding closed balls in $X_T^s$, thus 
\eee{
	\no{u-v}_{X_T^s} 
	=
	\no{S_u(u) - S_v(v)}_{X_T^s} 
\equiv 
	\no{S\left[ u_0-v_0, u_1-v_1,\alpha_u-\alpha_v,\beta_u - \beta_v ; -\p_x^2 (u^2-v^2)\right]}_{X_T^s}.
}
Hence, employing estimate \eqref{lin-est} and then the algebra property, we find
\eee{
	\no{u-v}_{X_{T}^s} 
	 \leq 
	c(s, T) 
	\left(
	\no{d_u-d_v}_{\mathcal D_T}
	+
	\sqrt {T}  \left[  \no{u}_{X_T^s}+ \no{v}_{X_T^s} \right] \no{u-v}_{X_T^s}\right).
}
By \eqref{duv}, we have $r_u, r_v \leq 2c(s, T) \left(\no{d}_{\mathcal D_T} + R\right)$. Therefore,
\eee{
\left[ 
	1- 4 c(s,T)^2 \sqrt T \left(\no{d}_{\mathcal D_T} + R\right)
\right]
\no{u-v}_{X_T^s} 
\leq
c(s,T) \no{d_u-d_v}_{\mathcal {D}_T}
}
which, in view of \eqref{contr-cond2} (which guarantees that the factor on the left side is strictly positive), can be rearranged to
\eee{
\no{u-v}_{X_T^s} 
\leq
\frac{c(s,T)}{1- 4 c(s,T)^2 \sqrt T \left(\no{d}_{\mathcal D_T} + R\right)}
 \no{d_u-d_v}_{\mathcal {D}_T}.
}
This inequality amounts to  local Lipschitz continuity of the data-to-solution map and concludes the proof of Theorem \ref{lwp-t} on the local Hadamard well-posedness of the Robin problem~\eqref{gbous} for the ``good'' Boussinesq equation on the half-line.

\section{Linear and nonlinear theory for $0\leq s < \frac 12$}
\label{l-nl-low-s}

In this section, we consider the case of low regularity corresponding to the range $0 \leq s < \frac 12$. In particular, we establish the linear estimate of Theorem \ref{linear-low-t}, which can then be combined with a contraction mapping argument to yield the local well-posedness result of Theorem \ref{lwp-t-low} for the nonlinear problem \eqref{gbous}.

In the range $0\leq s < \frac 12$, we consider problem \eqref{gbous} under the standard assumptions (e.g. see~\cite{l1993,x2008,farah2009,x2010}) $u_1(x) = v_1'(x)$ and $\beta(t) = b'(t)$, where $v_1 \in H^{s-1}(0, \infty)$ and $b \in H^{\frac{2s+1}{4}}(0, T)$. This additional structure allows us to employ the extension results of \cite[Theorem 11.4]{lm1972} and Lemma \ref{Chris-l}, which are valid for Sobolev spaces of exponent greater than $-\frac 12$. Indeed, this condition is not satisfied by the exponent $\frac{2s-3}{4}$ of the original boundary datum $\beta$ in our current range of $s<\frac 12$; however, it is satisfied by the exponent $\frac{2s+1}{4}$ of $b$. 
\\[2mm]
\textbf{Alternative decomposition and extensions.}  
Since we are in a setting of low regularity, we may decompose the forced linear problem as
\eee{\label{dec-low}
S\big[u_0, v_1', \alpha, b'; -\p_x^2 w\big](x, t)
=
S\big[u_0, v_1', \alpha, b'; 0\big](x, t)
+
\int_0^t S\big[0,-\p_x^2 w(\cdot, t'),0,0; 0\big](x, t-t') dt'.
}
We will further decompose the first term on the right side as follows. First, we consider extensions $U_0 \in H^s(\R)$ and $\mathcal U_1 \in H^{s-1}(\R)$ of  $u_0 \in H^s(0, \infty)$ and $v_1 \in H^{s-1}(0, \infty)$ such that 
\eee{\label{ic-ext2}
\no{U_0}_{H^s(\R)} \leq 2\no{u_0}_{H^s(0,\infty)},
\quad
\no{\mathcal U_1}_{H^{s-1}(\R)} \leq 2\no{v_1}_{H^{s-1}(0,\infty)}.
}
The following result follows directly from the solution formula \eqref{ivp-sol}:
\begin{lemma}\label{ivp-l}
Let $S[U_0, \mathcal U_1'; 0]$ denote the solution to the Cauchy problem \eqref{ivp} with data $U_0$ and $\mathcal U_1'$. Then, 
\eee{
\p_x^j S[U_0, \mathcal U_1'; 0](x, t) = \p_t S[V_{0, j}, V_{1, j}; 0](x, t),
\quad
j=1, 2,
}
where $V_{1, j} = \p_x^j U_0$ and $V_{0,j}$ is defined via its Fourier transform by
\eee{
\mathcal F\{V_{0, j}\}(k) = \frac{(ik)^{j-1}}{1+k^2} \mathcal F\{\mathcal U_1\}(k).
}
\end{lemma}
In view of Lemma \ref{ivp-l},  the first term on the right side of \eqref{dec-low} can be written as
\eee{\label{dec-low-h}
S\big[u_0, v_1', \alpha, b'; 0\big]
=
S\big[U_0, \mathcal U_1'; 0\big]\big|_{x>0}
+
S\big[0, 0, \varphi_0, q_0'; 0\big]
}
where
\eee{\label{psi0q0}
\varphi_0(t) = \alpha(t) - \left(\p_x + \gamma I\right)S\big[U_0, \mathcal U_1'; 0\big](0, t),
\quad
q_0(t) = b(t) - S[V_{0, 2}, V_{1, 2}; 0] (0, t) + \delta S[V_{0, 1}, V_{1, 1}; 0] (0, t).
}
The time estimates \eqref{ivp-te-t} and inequalities \eqref{ic-ext2} ensure that the datum $\varphi_0 \in H^{\frac{2s-1}{4}}(0, T)$ with the estimate
\eee{
\no{\varphi_0}_{H^{\frac{2s-1}{4}}(0, T)}
\lesssim
\no{\alpha}_{H^{\frac{2s-1}{4}}(0, T)} 
+
\no{u_0}_{H^s(0, \infty)}
+
\no{v_1}_{H^{s-1}(0, \infty)}, \quad s \in \R.
\label{phi01}
}
Moreover,  
\aaa{\label{psi0-est2}
\no{q_0}_{H^{\frac{2s+1}{4}}(0, T)}
&\leq
\no{b}_{H^{\frac{2s+1}{4}}(0, T)}
+
\no{S[V_{0, 2}, V_{1, 2}; 0] (0, \cdot)}_{H^{\frac{2s+1}{4}}(0, T)}
+
|\delta|
\no{S[V_{0, 1}, V_{1, 1}; 0] (0, \cdot)}_{H^{\frac{2s+1}{4}}(0, T)}
\nn\\
&\lesssim
\no{b}_{H^{\frac{2s+1}{4}}(0, T)}
 + \no{u_0}_{H^{s}(0, \infty)} + \no{v_1}_{H^{s-1}(0, \infty)}, \quad s \in \R,
}
with the last step in view of \eqref{ic-ext2}  and the fact that, by \eqref{ivp-te-t} and the definition of $V_{0, j}, V_{1, j}$, 
	\aaa{
		\no{S[V_{0,j}, V_{1,j}; 0](x)}_{H_t^{\frac{2s+1}{4}}(0, T)}^2 
		&\lesssim
		\no{V_{0,j}}_{H^{s}(\R)}^2 + \no{V_{1,j}}_{H^{s-2}(\R)}^2
		\nn \\
		& = 
			\int_\R \left(1+k^2\right)^{s-2} \left| k^{j-1}\mathcal F \{\mathcal U_1\} (k) \right|^2\, dk
		+
		\no
		{\p_x^j U_0}_{H^{s-2}(\R)}^2
		\nn\\
		&\leq
		\no{\mathcal U_1}_{H^{s-1}(\R)}^2 + \no{U_0}_{H^s(\R)}^2, 
		\quad j=1, 2, \ s \in \R. \label{8.5}
	} 

By  Lemma \ref{Chris-l}, for $ 0\leq s < \frac 12$   the extension $\Phi_0$ of $\varphi_0 \in H^{\frac{2s-1}{4}}(0, T)$ by $0$ outside $(0, T)$ belongs to $H^{\frac{2s-1}{4}}(\R)$ with the estimate
\eee{\label{phi0ext}
\no{\Phi_0}_{H^{\frac{2s-1}{4}}(\R)} 
\leq
c(T) \no{\varphi_0}_{H^{\frac{2s-1}{4}}(0, T)}.
}
In addition, by \cite[Theorem 11.4]{lm1972}, for $0 \leq s < \frac 12$ the extension $Q_0$ of $q_0$ by $0$ outside $(0, T)$ belongs to $H^{\frac{2s+1}{4}}(\R)$ with the estimate
\eee{\label{q0ext}
\no{Q_0}_{H^{\frac{2s+1}{4}}(\R)} \leq c(T) \no{q_0}_{H^{\frac{2s+1}{4}}(0, T)}.
}
In both of the above estimates, the constant $c(T)$ remains bounded as $T\to 0^+$ (see comment below \eqref{zero-ext}).

From \eqref{phi0ext} and \eqref{q0ext}, we may employ  Theorem \ref{r-t} with $\varphi(t) = \Phi_0(t)$ and $\psi(t) = Q_0'(t)$, which are permitted choices for the boundary data in that theorem thanks to the fact that they are supported inside $(0, T)$, to obtain
\eee{
\no{S\big[0, 0, \Phi_0, Q_0'; 0\big]}_{X_T^s} 
\leq
c(s, T) \left(\no{\varphi_0}_{H^{\frac{2s-1}{4}}(0, T)} + \no{q_0}_{H^{\frac{2s+1}{4}}(0, T)}\right),  \quad 0\leq s < \frac 12, \ t\in (0, T).
}
Noting that $S\big[0, 0, \Phi_0, Q_0'; 0\big] \equiv S\big[0, 0, \varphi_0, q_0'; 0\big]$ on $(x, t) \in (0, \infty) \times (0, T)$ and using the estimates \eqref{phi01} and \eqref{psi0-est2}, we infer 
\ddd{\label{ge1}
\no{S\big[0, 0, \varphi_0, q_0'; 0\big]}_{X_T^s}
\leq
c(s, T) \left(
\no{u_0}_{H^{s}(0, \infty)} + \no{v_1}_{H^{s-1}(0, \infty)}
+
\no{\alpha}_{H^{\frac{2s-1}{4}}(0, T)}
+
\no{b}_{H^{\frac{2s+1}{4}}(0, T)}
\right)
}
for all $0 \leq s < \frac 12$ and $t\in (0, T)$.
Returning to the decomposition \eqref{dec-low-h}, combining estimate \eqref{ge1}with the space estimate \eqref{ivp-se-t}  for the homogeneous Cauchy problem solution $S\big[U_0, \mathcal U_1'; 0\big]$ and the extension inequalities \eqref{ic-ext2}, we obtain
\eee{\label{hibvp-low}
\no{S\big[u_0, v_1', \alpha, b'; 0\big]}_{X_T^s}
\leq
c(s, T) \left(
\no{u_0}_{H^{s}(0, \infty)} + \no{v_1}_{H^{s-1}(0, \infty)}
+
\no{\alpha}_{H^{\frac{2s-1}{4}}(0, T)}
+
\no{b}_{H^{\frac{2s+1}{4}}(0, T)}
\right)
}
for all $0 \leq  s < \frac 12$ and $t\in (0, T)$.

Next, differentiating \eqref{dec-low} in $t$ while observing that $S\big[0,-\p_x^2 w(\cdot, t'),0,0; 0\big](x, t-t')\big|_{t'=t} = 0$, we have
\eee{
\no{S\big[u_0, v_1', \alpha, b'; -\p_x^2 w\big]}_{X_T^s}
\leq
\no{S\big[u_0, v_1', \alpha, b'; 0\big]}_{X_T^s}
+
\int_0^T \no{S\big[0,-\p_x^2 w(t'),0,0; 0\big](\cdot-t')}_{X_T^s} dt'.
}
The first term on the right side can be handled via the estimate \eqref{hibvp-low}, which can also be employed for the integrand of the $t'$ integral to yield
\eee{\label{hibvp-low2}
\no{S\big[0,-\p_x^2 w(t'),0,0; 0\big](\cdot-t')}_{X_T^s}
\leq
c(s, T) \no{w(t')}_{H_x^s(0, \infty)},  \quad
0 \leq s < \frac 12, \ t'\in (0, T). 
}
Note importantly that we can employ \eqref{hibvp-low} for $S\big[0,-\p_x^2 w(\cdot, t'),0,0; 0\big](x, t-t')$ because it does not require any compatibility conditions between the data, as we are in a low regularity setting. This would not have been the case in the higher regularity setting of $s>\frac 12$, which is why the decomposition \eqref{dec} was used instead of \eqref{dec-low}.
Combining \eqref{dec-low} with \eqref{hibvp-low}-\eqref{hibvp-low2}, we deduce 
\aaa{ 
\no{S\big[u_0, v_1', \alpha, b'; -\p_x^2 w\big]}_{X_T^s}
&\leq
c(s, T) \Big(
\no{u_0}_{H^{s}(0, \infty)} + \no{v_1}_{H^{s-1}(0, \infty)}
+
\no{\alpha}_{H^{\frac{2s-1}{4}}(0, T)}
+
\no{b}_{H^{\frac{2s+1}{4}}(0, T)}
\nn\\
&\hspace*{1.8cm}
+
\no{w}_{L_t^1((0, T); H_x^s(0, \infty))} 
 \Big), \quad 0 \leq s < \frac 12,
  \label{834}
}
which corresponds to the $X_T^s$ part of the $Y_T^s$ norm of the linear estimate \eqref{lin-est-low} in Theorem \ref{linear-low-t}.
\\[2mm]
\textbf{Motivation for the solution space $Y_T^s$.}  
Attempting to use the linear estimate \eqref{834} for the analysis of the nonlinear problem \eqref{gbous}, we have $w = u^2$ so we need to handle  the norm $\no{u^2(t)}_{H_x^s(0, \infty)}$.   
Since $s\geq 0$, this can be done via the fractional Leibniz rule (see \cite{kp1988,cw1991,bmmn2014} and, in particular, \cite[Theorem 1]{go2014}) along with the zero extension Theorem 11.4 in \cite{lm1972}, resulting in
\eee{\label{u20}
\no{u^2(t)}_{H_x^s(0, \infty)}
\leq
\no{u(t)}_{L_x^\infty(0, \infty)} \no{u(t)}_{H_x^s(0, \infty)}, \quad 0 \leq s<\frac12, \ t \in (0, T).
}
Therefore, by also applying the Cauchy-Schwarz inequality in $t$,  \eqref{834} yields
\ddd{
\no{S\big[u_0, v_1', \alpha, b'; -\p_x^2(u^2)\big]}_{X_T^s}
&\leq
c(s, T) \Big(
\no{u_0}_{H^{s}(0, \infty)} + \no{v_1}_{H^{s-1}(0, \infty)}
+
\no{\alpha}_{H^{\frac{2s-1}{4}}(0, T)}
+
\no{b}_{H^{\frac{2s+1}{4}}(0, T)} \label{838} \\
&\hspace*{1.8cm}
+
T^{\frac14}
\no{u}_{L_t^\infty((0, T); H_x^s(0, \infty))}
\no{u}_{L_t^4((0, T); L_x^\infty(0, \infty))} 
 \Big), \quad 0 \leq s < \frac 12.
}
Estimate \eqref{838} motivates the solution space $Y_T^s$ instead of $X_T^s$ when $s<\frac 12$. We note that, since $t \in (0, T)$ is finite, we could also have $L_t^2$ in place of $L_t^4$ in the definition \eqref{yst} of $Y_T^s$; however, as we shall see below, the space $L_t^4$ is a more natural choice as it can also accommodate $t\in \R$. 
In order to complete the proof of Theorem~\ref{linear-low-t}, it remains to establish the analogue of \eqref{834} in the new space $L_t^4((0, T); L_x^\infty(0, \infty))$ introduced via~\eqref{838}.
\\[2mm]
\textbf{Additional linear estimates in $L_t^4((0, T); L_x^\infty(0, \infty))$.}
We begin with the homogeneous linear Cauchy problem~\eqref{ivp}, whose solution formula $U = S[U_0,0;0] + S[0, U_1;0]$ is given by \eqref{ivp-sol}. By \cite[Lemma 2.4]{l1993}, 
\eee{\label{S+est}
\no{S[U_0,0;0]}_{L_t^4((0, T); L_x^\infty(0,\infty))} 
 \leq c \, (1+T^{\frac14}) \no{U_0}_{L^2(0,\infty)}.
}
Moreover, since 
\eee{
S[0, U_1;0](x, t)  =
	\frac1{2\pi } \int_\R \frac{e^{ikx}}{2i\omega}\left( e^{i\omega t}-  e^{-i\omega t}\right)\mathcal F\{U_1\}(k)\,dk
}
we adapt the proof of \cite[Lemma 2.5]{l1993} in order to exploit the removable singularity of the above integrand at $k=0$ and hence avoid the need to assume $U_1 = \mathcal U_1'$ at this very point (although this assumption is otherwise made due to the extensions discussed earlier). 

Specifically, using the smooth cut-off function $\theta \in C_c^\infty(\R)$ such that $0 \leq \theta(t) \leq 1$ for all $t\in \R$, $\theta(t) \equiv 1$ on $[- 1, 1]$, and $\theta(t) \equiv 0$ on $[-2, 2]^c$, we write 
\eee{
S[0,U_1;0] 
=
S_{\text{far}}^+[0,U_1;0] 
-S_{\text{far}}^-[0,U_1;0] + \left(S_{\text{near}}^+[0,U_1;0]  - S_{\text{near}}^-[0,U_1;0]\right)
}
where
\aaa{
S_{\text{far}}^\pm[0,U_1;0](x, t)
&:=
\frac1{2\pi} \int_\R e^{ikx\pm i\omega t} \left|\omega''\right|^{\frac 14} 
\left(1-\theta(k)\right)
\frac{\mathcal F\{U_1\}(k)}{2i\omega \left|\omega''\right|^{\frac 14}} \,dk,
\\
S_{\text{near}}^\pm[0,U_1;0](x, t)
&:=
\frac1{2\pi} \int_\R e^{ikx\pm i\omega t} \theta(k)  \frac{\mathcal F\{U_1\}(k)}{2i\omega}\,dk.
}
Since $\omega'' = k\left(3+2k^2\right) \left(1+k^2\right)^{-\frac 32}$, 
we can employ \cite[Theorem 2.1]{kpv1991-osc} to infer 
\eee{\label{sf}
\no{S_{\text{far}}^\pm[0,U_1;0]}_{L_t^4((0, T); L_x^\infty(\R))}
\lesssim
\Big\|\left(1-\theta(k)\right) \frac{\mathcal F\{U_1\}(k)}{\omega \left|\omega''\right|^{\frac 14}}\Big\|_{L^2(\R)}
\lesssim
\no{U_1}_{H^{-2}(\R)}
}
after also noting that $5 \cdot 2^{-\frac32} \leq |\omega''(k)| \leq 5$ for $|k| \geq 1$.
In addition, by the Sobolev embedding theorem and the fact that $\frac{\sin^2(\omega t)}{\omega^2} \leq t^2$,
\aaa{
\no{S_{\text{near}}^+[0,U_1;0](t)-S_{\text{near}}^-[0,U_1;0](t)}_{L_x^\infty(\R)}^2
&\leq
\no{S_{\text{near}}^+[0,U_1;0](t)-S_{\text{near}}^-[0,U_1;0](t)}_{H_x^1(\R)}^2
\nn\\
&=
\int_{|k|\leq 2} \left(1+k^2\right) \left|\theta(k)\right|^2 \frac{\sin^2(\omega t)}{\omega^2}  \left|\mathcal F\{U_1\}(k)\right|^2 dk
\nn\\
&\leq
\int_{|k|\leq 2} \left(1+k^2\right) \left|\theta(k)\right|^2 t^2 \left|\mathcal F\{U_1\}(k)\right|^2 dk
\lesssim
t^2 \no{U_1}_{H^{-2}(\R)}^2.\label{sn}
}
Combining \eqref{sf} and \eqref{sn}, we have
\eee{\label{su1-est2}
\no{S[0,U_1;0]}_{L_t^4((0, T); L_x^\infty(\R))}
\lesssim
\no{U_1}_{H^{-2}(\R)}
}
which together with estimate \eqref{S+est} yields
\eee{\label{ivp-est-low}
\no{U}_{L_t^4((0, T); L_x^\infty(\R))}
\lesssim
\no{U_0}_{L^2(\R)} + \no{U_1}_{H^{-2}(\R)} 
 \leq
\no{U_0}_{H^s(\R)} + \no{U_1}_{H^{s-2}(\R)}, \quad s\geq 0.
}

We move on to the forced Cauchy problem \eqref{flb-ivp}, whose solution  is given by \eqref{W-sol}.
By Minkowski's integral inequality and estimate \eqref{su1-est2} with $U_1 = F(t')$,  
\aaa{
\no{Z}_{L_t^4((0, T); L_x^\infty(\R))} 
& \leq
  \left[ \int_0^T \left(\int_0^t \no{S[0,F(\cdot,t');0](\cdot,t-t')}_{L_x^\infty(\R)}\,dt'\right)^4dt \right]^\frac14 \nn\\
  & \leq
  \int_0^T \left( \int_{t'}^T \no{S[0,F(\cdot,t');0](\cdot,t-t')}_{L_x^\infty(\R)}^4 dt\right)^\frac14 dt' \nn\\
  &= \int_0^T \no{S[0,F(t');0]}_{L^4_\tau((0,T-t');L_x^\infty(\R))} dt'
  \lesssim \no{F}_{L_t^1((0, T); H_x^{-2}(\R))}.
   \label{mink-t}
}

Next, we turn our attention to the reduced initial-boundary value problem \eqref{lgbous-r}, whose solution \eqref{v-sol} is the sum of the terms \eqref{iterms}. 
We only provide the details for $I_{\varphi, 1}$, since the estimation of the other three terms is entirely analogous.
We begin with the portion of $I_{\varphi,1}$ along the real axis, namely the integral $I_{\varphi,1}^r$ given by \eqref{ifr-def}. Since this term makes sense for all $x\in \R$, we can readily employ the Cauchy problem estimate  \eqref{S+est}  along with Plancherel's theorem to deduce
\aaa{
\no{I_{\varphi,1}^r}_{L_t^4([0,T];L_x^\infty(\R))}
&\lesssim
(1+T^{\frac14}) 
\no{\frac{k(k+\nu(k))}{i\omega(k)\Delta_1(k)}\cdot
	\nu(k)(\nu(k)-i\delta)\mathcal F\{\varphi\}(-\omega(k))}_{L^2(a,\infty)}
\nn\\
&\lesssim
(1+T^{\frac14})  \no{\varphi}_{H^{-\frac14}(\R)}
\label{irest}
}
with the last inequality following along the lines of \eqref{redI-r1}-\eqref{redI-r2} with $s=0$.

Moving on to the imaginary axis and the term $I^i_{\varphi,1}$ given by \eqref{im-phi}, we recall that this term does not make sense for $x<0$ so the approach used for $I^r_{\varphi,1}$ cannot be applied here. Instead, we observe that
\eee{
I^i_{\varphi,1}(x,t) 
= 
\int_a^\infty e^{-\lambda x-i\check{\omega}t}\mathcal F\{f\}(\lambda)\,d\lambda
}
where, for $\check{\omega} = \check{\omega}(\lambda) :=\omega(i\lambda)$ and $\check \nu, \check \Delta_1$ defined analogously, the function $f$ is defined via its Fourier transform 
\eee{
\mathcal F\{f\}(\lambda) := 
\chi_{(a, \infty)}(\lambda) \, \frac{\lambda(i\lambda+\check \nu)}{i\check \omega \check \Delta_1} \, \check \nu (\check \nu-i\delta) \mathcal{F}\{\varphi\} (-\check \omega).
}
Writing 
\eee{
I_{\varphi,1}^i(x,t) 
=
 \int_\R f(y)  K(x,y,t) dy, 
 \quad
K(x,y,t) := \int_a^\infty e^{-\lambda x - i\check{\omega}t-i\lambda y}\,d\lambda
}
and using the duality norm combined with the Cauchy-Schwarz inequality, we have
\aaa{
\no{I^i_{\varphi,1}}_{L_t^4(\R;L_x^\infty(0, \infty))} 
&=
 \sup_{\no{\eta}_{L_t^{4/3}L_x^1}=1}
 \left|
 \int_\R f(y) \int_\R \int_0^\infty K(x,y,t)\eta(x,t)\,dx\,dt\,dy
 \right| \nn \\
 & \leq 
(\sup_{\eta} M) \no{f}_{L^2(\R)}, 
\quad
M := \no{\int_\R \int_0^\infty K(x,y,t) \eta(x,t)\,dx\,dt}_{L_y^2(\R)}.
\label{C-S}
}
Now, by Plancherel's theorem and the calculations \eqref{L2-norm}-\eqref{rvp-est} for $j=0$, we have $\no{f}_{L^2(\R)} \lesssim \no{\varphi}_{H^{-\frac14}(\R)}$ which in turn implies
\eee{\label{632}
\no{I^i_{\varphi,1}}_{L_t^4(\R;L_x^\infty(0, \infty))} \lesssim ( \sup_{\eta} M) \no{\varphi}_{H^{-\frac14}(\R)}.
}
It remains to show that $\displaystyle \sup_\eta M$ is finite. Keeping in mind that $\check\omega \in \R$, we consider
\eee{
M^2 
=
 \int_\R \int_\R \int_0^\infty \int_0^\infty 
\eta(x,t) \overline{\eta(z,\tau)} N(x,z,t,\tau)\,dz\,dx\,d\tau\,dt,
}
where 
$N(x,z,t,\tau) := \displaystyle \int_\R K(x,y,t) \overline{K(z,y,\tau)}\,dy$. 
Now, since $K(x,y,t) = \mathcal F_\lambda \{e^{-\lambda x -i\check{\omega}t}\chi_{(a,\infty)}(\lambda)\}(y)$, 
\aaa{
N(x,z,t,\tau)
& =
\int_a^\infty e^{-\lambda z+i\check{\omega}\tau} \int_\R e^{i\lambda y}K(x,y,t)\,dy\,d\lambda
\nn\\
& \simeq 
\int_a^\infty e^{-\lambda z+i\check{\omega}\tau} \cdot e^{-\lambda x-i\check{\omega}t}\,d\lambda
= 
\wtil{N}(x+z,t-\tau),
}
where $\displaystyle \wtil N (x,t) := \int_a^\infty e^{-\lambda x-i\check{\omega}\tau}\,d\lambda$. Hence, using also H\"older's inequality in $t$, we find
\aaa{
M^2 
& \lesssim 
\int_\R 
 \left( \int_\R
	\sup_x \left|
		\int_0^\infty \overline{\eta(z,\tau)} \wtil N(x+z,t-\tau)\,dz
	\right|
d\tau 
\right) 
\no{\eta(t)}_{L_x^1(0, \infty)}\, dt 
\nn\\
& \leq 
\no{\int_\R \sup_{x,z} |\wtil N(x+z,t-\tau)| \no{\eta(\tau)}_{L_z^1(0, \infty)} \,d\tau}_{L_t^4(\R)}.
\label{m2bound}
}
Consider $N_b(x,t) := \displaystyle \int_a^b e^{-\lambda x -i \check\omega t}\,d\lambda$ and note that $ |(-\check\omega)''| \to 2 \text{ as } \lambda \to \infty$ thus for large enough $a$ we have $ |(-\check\omega)''| \geq 1 $ for all $\lambda \in [a, b]$. Hence, by the van der Corput Lemma \cite[page 334]{s1993}, 
\eee{
|N_b(x,t)| \leq c |t|^{-\frac12} \left( e^{-bx} + \int_a^b  xe^{-\lambda x}\,d\lambda \right)
}
where the constant $c$ is independent of $a$ and $b$. In fact, since 
since $x,a,b \geq 0$, it follows that $ |N_b(x,t)| \leq 2c |t|^{-\frac12}$ and taking the limit $b \to \infty$ we  infer
\eee{
	|\wtil N (x+z,t-\tau)| \lesssim |t-\tau|^{-\frac12}.
}
Combining this bound with \eqref{m2bound} and the Sobolev-Hardy-Littlewood inequality \cite[Theorem 2.6]{lp2009}, we conclude that
\eee{\label{m2}
M^2 \lesssim \no{\int_\R |t-\tau|^{-\frac12} \no{\eta(\tau)}_{L_z^1(0, \infty)} d\tau}_{L_t^4(\R)}
\lesssim \no{\eta}_{L_\tau^{4/3}(\R; L_z^1(0, \infty))} =1, 
}
which completes the estimation of $I_{\varphi, 1}^i$ in view of \eqref{632}. It is important to note that the space $L_t^4(\R)$ is the only choice that allows us to combine the duality norm with the Sobolev-Hardy-Littlewood inequality as shown above. In this sense, $L_t^4(0, T)$ is the natural choice of space to include in  the definition \eqref{yst}.

Finally, we consider the term $I_{\varphi, 1}^C$ given by \eqref{234}, which corresponds to the integral along the finite quarter-circle depicted in Figure \ref{fig:0}. We have 
\aaa{
	\no{I_{\varphi,1}^C(t)}_{L_x^\infty(0,\infty)} 
	& \lesssim
	a \sup_{x \in (0, \infty)} 
	\int_0^{\frac{\pi}{2}} e^{\text{Im}(\omega(ae^{i\theta})t)} e^{-ax\sin \theta}|\mathcal F\{\varphi\}(-\omega(ae^{i\theta}))|\,d\theta \nn \\
	& \leq
	ae^{a\sqrt{1+a^2}T} 
	\int_0^{\frac{\pi}{2}} |\mathcal F\{\varphi\}(-\omega(ae^{i\theta}))|\,d\theta
}
and by the earlier calculations from \eqref{eta}-\eqref{Conv2}, we obtain
\eee{
\no{I_{\varphi,1}^C(t)}_{L_x^\infty(0,\infty)} 
 \leq
	c(a,d,T) \no{\varphi}_{H^{-d}(\R)}, \quad d \in \N, 
}
where the constant 
\eee{
c(a,d,T) =  
	\frac{\pi a^{2d+1}e^{a\sqrt{1+a^2}T}}{2(1-e^{-(T+1)})^d}\sqrt{(d+1)T+d}\, \big(1+e^{-(T+1)}\big)^d
}
 remains bounded as $T\to 0^+$. Thus,  
\eee{\label{icest}
	\no{I_{\varphi,1}^C}_{L^4_t((0, T);L_x^\infty(0,\infty))} \lesssim c(a,d,T) \, T^{\frac 14} \no{\varphi}_{H^{-d}(\R)}, \quad d \in \N.
}

Overall, the estimates \eqref{irest}, \eqref{632},  \eqref{m2} and \eqref{icest} yield 
\eee{
	\no{I_{\varphi,1}}_{L^4_t((0, T);L_x^\infty(0,\infty))} 
	 \lesssim
	 c(a,T)
	\no{\varphi}_{H^{-\frac 14}(\R)}, 	
}
where $c(a, T)$ remains bounded as $T\to 0^+$. Estimating the remaining terms in \eqref{iterms} analogously, we obtain 
\eee{\label{red-low}
	\no{v}_{L^4_t((0, T);L_x^\infty(0,\infty))} 
	\lesssim 
	 c(a,T) 
	 \left( 
	 	\no{\varphi}_{H^{-\frac14}(\R)} + \no{\psi}_{H^{-\frac34}(\R)}
	 \right).
}
Estimate \eqref{red-low} together with the homogeneous Cauchy estimate \eqref{ivp-est-low} for $U_1 = \mathcal U_1'$ defined by \eqref{ic-ext2} as well as the forced Cauchy estimate~\eqref{mink-t} for $F = -\p_x^2 W$, where $W(t) \in H_x^s(\R)$ is the zero extension of $w(t) \in H_x^s(0, \infty)$, results in the linear estimate \eqref{lin-est-low}, completing the proof of Theorem \ref{linear-low-t} for the forced linear problem \eqref{lgbous}. 
Theorem~\ref{lwp-t-low} for the local well-posedness of the nonlinear problem \eqref{gbous} in the low regularity range of $0\leq s<\frac 12$ then follows by combining the linear estimate \eqref{lin-est-low} with a contraction mapping argument entirely analogous to the one presented for the case of $s>\frac 12$ in Section \ref{lwp-s}.

\section{Linear problem solution via Fokas's unified transform}
\label{ut-s}

In this section, we solve problem \eqref{lgbous} for the forced linear ``good'' Boussinesq equation on the half-line with nonzero Robin boundary data. We derive a novel solution formula via the unified transform (also known as the Fokas method), which provides the direct analogue of the Fourier transform in the case of spatial domains that involve a boundary \cite{f1997,f2008}.
For the purpose of this section, we work under the assumption of initial and boundary data with sufficient smoothness and decay at infinity. 

For any $\phi \in L^2(0, \infty)$, we consider the half-line Fourier transform pair
\eee{\label{hl-ft}
\what \phi(k) = \int_0^\infty e^{-i k x} \phi(x) \,dx, 
\quad \text{Im}(k) \leq 0,
\qquad
\phi(x) = \frac{1}{2\pi} \int_{\mathbb R} e^{i k x} \what \phi(k) \,dk, 
\quad
x\in (0, \infty),
}
where we note that $\what \phi$ is analytic in the lower half of the complex $k$-plane via a Paley-Wiener theorem \cite[Theorem 7.2.4]{s1994}.
Taking the half-line Fourier transform of equation \eqref{lgbous} 
and using the boundary conditions in~\eqref{lgbous} to eliminate $u_x(0, t)$ and $u_{xx}(0, t)$, we have
\ddd{\label{ode}
\p_t^2 \what u(k, t) + \omega(k)^2 \,  \what u( k, t) &= 
 \what{f}(k,t)+g_3(t)
 +
 \left[i \gamma \delta k + \gamma\left(1+k^2\right) -ik\left(1+k^2\right)\right] g_0(t)
 \\
 &\quad
 + ik \beta(t) - \left[i \delta k + \left(1+k^2\right)\right]\alpha(t)
}
where we have introduced the notation $g_j(t) := \p_x^j u(0,t)$, $j=0, 3$ and  the dispersion relation
\ddd{\label{om}
\omega(k):= 
k\left(1+k^2\right)^{\frac12}, \quad k \in \C \setminus \mathcal{B},
}
is made single-valued via a branch cut along the line segment $\mathcal{B} = i[-1,1]$ (red segment in Figure \ref{fig:1}) and defining
\eee{\label{bcut}
\left(1+k^2\right)^{\frac12}
=
\sqrt{ \left|1+k^2\right|} \, e^{i(\theta_1+\theta_2-\pi)/2}, \quad \theta_1, \theta_2 \in [0, 2\pi], 
}
with the angles $\theta_{1, 2}$ corresponding to $k\pm i$ and measured counterclockwise, starting below the branch points $\mp i$.
Based on this definition, for $k\in \R$ we have $\omega(k) = |k| \sqrt{|1+k^2|}$ implying that $\omega$ is continuous from either side of $k=0$ along the real axis. 

Solving the second-order linear constant-coefficient ODE \eqref{ode} via variation of parameters while taking into account the initial conditions in \eqref{lgbous}, we find the following spectral identity which is known as the global relation:
\aaa{\label{gr0}
&\what u (k,t)
= \frac{e^{i\omega t}}{2i\omega} \left(\what{u}_1(k)+ i\omega \what{u}_0(k)\right) - \frac{e^{-i\omega t}}{2i\omega} \left( \what{u}_1(k)-i\omega \what{u}_0(k)\right) +  \frac{e^{i\omega t} }{2i\omega}\int_0^t e^{-i\omega\tau} \what f(k,\tau)\,d\tau -\frac{e^{-i\omega t}}{2i\omega} \int_0^t e^{i\omega\tau} \what f(k,\tau)\,d\tau \nn\\
& \quad + \frac{e^{i\omega t}}{2i\omega} \left\{  \wtil g_3(\omega, t)
 +
 \left[i \gamma \delta k + \gamma\left(1+k^2\right) -ik\left(1+k^2\right)\right] \wtil g_0(\omega, t)
 + ik \wtil \beta(\omega, t) - \left[i \delta k + \left(1+k^2\right)\right] \wtil \alpha(\omega, t)\right\} \\
& \quad -\frac{e^{-i\omega t}}{2i\omega}  \left\{  \wtil g_3(-\omega, t)
 +
 \left[i \gamma \delta k + \gamma\left(1+k^2\right) -ik\left(1+k^2\right)\right] \wtil g_0(-\omega, t)
 + ik \wtil \beta(-\omega, t) - \left[i \delta k + \left(1+k^2\right)\right] \wtil \alpha(-\omega, t)\right\}, \nn
}
where $\wtil{g}(\omega, t) := \displaystyle \int_0^t e^{-i\omega\tau} g(\tau)\,d\tau$. 
Then, inverting via \eqref{hl-ft}, we arrive at the integral representation 
\aaa{\label{ir}
&2\pi u(x,t)=   \int_{\R} \frac{e^{ikx}}{2i\omega}\left\{e^{i\omega t}\left[\what{u}_1(k)+ i\omega \what{u}_0(k)\right\}-e^{-i\omega t}\left\{ \what{u}_1(k)-i\omega \what{u}_0(k)\right]\right\}dk \\
&  +  \int_{\R} \frac{e^{ikx}}{2i\omega} \left\{ e^{i\omega t } \int_0^t e^{-i\omega\tau} \what f(k,\tau)\,d\tau - e^{-i\omega t} \int_0^t e^{i\omega\tau} \what f(k,\tau)\,d\tau  \right\} dk \nn \\
&  +  \int_{\R} \frac{e^{ikx}}{2i\omega} 
\bigg(
e^{i\omega t} \left\{  \wtil g_3(\omega, t)
 +
 \left[i \gamma \delta k + \gamma\left(1+k^2\right) -ik\left(1+k^2\right)\right] \wtil g_0(\omega, t)
 + ik \wtil \beta(\omega, t) - \left[i \delta k + \left(1+k^2\right)\right] \wtil \alpha(\omega, t)\right\}  \nn \\
&  -  
e^{-i\omega t} \left\{  \wtil g_3(-\omega, t)
 +
 \left[i \gamma \delta k + \gamma\left(1+k^2\right) -ik\left(1+k^2\right)\right] \wtil g_0(-\omega, t)
 + ik \wtil \beta(-\omega, t) - \left[i \delta k + \left(1+k^2\right)\right] \wtil \alpha(-\omega, t)\right\} \bigg) dk, \nn 
}
which is not a solution formula as it contains two unknown quantities, namely the transforms $\wtil{g}_3$ and $\wtil{g}_0$ of the unknown boundary values $u_{xxx}(0, t)$ and $u(0, t)$. However, these unknown transforms can be eliminated from \eqref{ir} by combining certain symmetries of $\omega$ with appropriate deformations in the complex plane.

\setlength{\multicolsep}{1mm}
\setlength{\columnsep}{-3mm}
\begin{multicols}{2}
First, we note that the integrand of the last $k$ integral in \eqref{ir} is entire. For example,  the term
$$
\frac{e^{i\omega t}\wtil g_3(\omega,t)-e^{-i\omega t}\wtil g_3(-\omega,t)}{2i\omega}
=
\int_0^t  \frac{\sin(\omega(t-t'))}{\omega} \, g_3( t') dt'
$$
is continuous across $\mathcal B$ since $\omega$   changes sign across $\mathcal B$ and the sine function is odd. Thus, we can employ Cauchy's theorem to deform the contour of integration of the last $k$ integral in \eqref{ir} from $\R$ to a complex contour $\Gamma$ bypassing $\mathcal B$ from above, as shown in Figure \ref{fig:1}.
\columnbreak

\begin{Figure}
\begin{center}
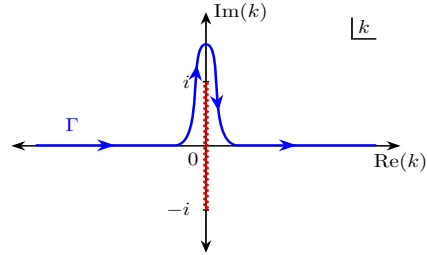

\resizebox{5.8cm}{3.4cm}{
\begin{tikzpicture}[>=Stealth]
    \draw[<->, thick] (-3.2,0) -- (3.2,0) node[below] {Re($k$)};
    \draw[<->, thick] (0,-2) -- (0,2.5) node[right] {Im($k$)};
    \node[below left] at (0,0) {0};
        \draw[thick] (-0.05, 1.2) -- (0.05, 1.2) node[left=5pt] {$i$}; 
     \draw[thick] (-0.05, -1.2) -- (0.05, -1.2) node[left=5pt] {$-i$}; 
      \draw[red, thick, branch cut] (0,1.2) -- (0,-1.2);
    \draw[blue, very thick, decoration={
        markings,
        mark=at position 0.15 with {\arrow{>}},
        mark=at position 0.45 with {\arrow{>}},
        mark=at position 0.65 with {\arrow{>}},
        mark=at position 0.85 with {\arrow{>}}},
        postaction={decorate}]
        (-2.8, 0.01) -- (-0.5, 0.01) 
        .. controls (-0.02, 0.05) and (-0.3, 1.9) .. (0, 1.9) 
        .. controls (0.3, 1.9) and (0.01, 0.1) .. (0.5, 0.01) 
        -- (2.8, 0.01); 
    \node[blue] at (-2.2, 0.4) {$\Gamma$};
    \begin{scope}[shift={(2.4,2.0)}]
        \draw[thick] (0,0.4) -- (0,0) -- (0.4,0);
        \node at (0.2, 0.24) {$k$};
    \end{scope}
\end{tikzpicture}
}
\end{center}
\vspace*{-2.5mm}
\captionof{figure}{Deformation from $\R$ to $\Gamma$.}
 \label{fig:1} 
\end{Figure}
\end{multicols}

Next, we observe that  $\omega(-k) = \omega(k)$ and $\omega(\pm \nu(k)) = -\omega(k)$, where $\nu(k) : = i \left(1+k^2\right)^{\frac12}$. Moreover, by the definition \eqref{bcut}, for $|k|\gg 1$ we have $\nu(k)\simeq ik$ and $\omega(k) \simeq k^2$ so, in turn,  
$
\text{Im}(\nu)
\simeq
\text{Re}(k)
$
and
$\text{Re}(i\omega(k)) 
\simeq 
-2 \text{Re}(k)\, \text{Im}(k)$.
Hence, combining Cauchy's theorem with exponential decay, we further deform the contour $\Gamma$ to either $\p D_1$ (in the case of the transforms involving $-\omega$) or $\p D_2$ (in the case of the transforms involving $\omega$), as shown in Figure \ref{fig:0}, resulting in 
\aaa{\label{IR}
&u(x,t)= \frac{1}{2\pi} \int_{\R} \frac{e^{ikx}}{2i\omega}\left\{e^{i\omega t}\left[\what{u}_1(k)+ i\omega \what{u}_0(k)\right\}-e^{-i\omega t}\left\{ \what{u}_1(k)-i\omega \what{u}_0(k)\right]\right\}dk \\
& \quad +  \frac{1}{2\pi} \int_{\R} \frac{e^{ikx}}{2i\omega} \left\{ e^{i\omega t } \int_0^t e^{-i\omega\tau} \what f(k,\tau)\,d\tau - e^{-i\omega t} \int_0^t e^{i\omega\tau} \what f(k,\tau)\,d\tau  \right\}dk \nn \\
& \quad +\frac{1}{2\pi} \int_{\p D_2} \frac{e^{ikx+i\omega t}}{2i\omega} \left\{  \wtil g_3(\omega, t)
 +
 \left(i \gamma \delta k - \gamma \nu^2 + ik\nu^2\right) \wtil g_0(\omega, t)
 + ik \wtil \beta(\omega, t) - \left(i \delta k -\nu^2\right) \wtil \alpha(\omega, t)\right\} dk \nn \\
& \quad - \frac{1}{2\pi} \int_{\p D_1} \frac{e^{ikx-i\omega t}}{2i\omega}  \left\{  \wtil g_3(-\omega, t)
 +
 \left(i \gamma \delta k - \gamma \nu^2 + ik\nu^2\right) \wtil g_0(-\omega, t)
 + ik \wtil \beta(-\omega, t) - \left(i \delta k - \nu^2\right) \wtil \alpha(-\omega, t)\right\} dk. \nn 
}

Now, combining the aforementioned symmetries of $\omega(k)$ with the global relation \eqref{gr0} while noting that $\nu(-k) = -\nu(k)$ and $\nu(\pm \nu(k)) = \mp k$, we obtain the following three additional spectral identities:
\aaa{\label{gr-k}
\what u (-k,t)
&= \frac{e^{i\omega t}}{2i\omega} \left(\what{u}_1(-k)+ i\omega \what{u}_0(-k)\right) - \frac{e^{-i\omega t}}{2i\omega} \left( \what{u}_1(-k)-i\omega \what{u}_0(-k)\right) 
\\
&\quad
+  \frac{e^{i\omega t} }{2i\omega}\int_0^t e^{-i\omega\tau} \what f(-k,\tau)\,d\tau -\frac{e^{-i\omega t}}{2i\omega} \int_0^t e^{i\omega\tau} \what f(-k,\tau)\,d\tau 
\nn\\
& \quad + \frac{e^{i\omega t}}{2i\omega} \left[  \wtil g_3(\omega, t)
-
 \left(i \gamma \delta k + \gamma \nu^2 + ik \nu^2\right) \wtil g_0(\omega, t)
 - ik \wtil \beta(\omega, t) + \left(i \delta k +\nu^2\right) \wtil \alpha(\omega, t)\right] \nn\\
& \quad -\frac{e^{-i\omega t}}{2i\omega}  \left[  \wtil g_3(-\omega, t)
-
 \left(i \gamma \delta k + \gamma \nu^2 +ik\nu^2\right) \wtil g_0(-\omega, t)
 - ik \wtil \beta(-\omega, t) + \left(i \delta k +\nu^2\right) \wtil \alpha(-\omega, t)\right], \quad \text{Im}(k)\geq 0,
 \nn\\[2mm]
\label{gr+nu}
\what u (\nu,t)
&= \frac{e^{i\omega t}}{2i\omega} \left(\what{u}_1(\nu)+ i\omega \what{u}_0(\nu)\right) - \frac{e^{-i\omega t}}{2i\omega} \left( \what{u}_1(\nu)-i\omega \what{u}_0(\nu)\right) +  \frac{e^{i\omega t} }{2i\omega}\int_0^t e^{-i\omega\tau} \what f(\nu,\tau)\,d\tau -\frac{e^{-i\omega t}}{2i\omega} \int_0^t e^{i\omega\tau} \what f(\nu,\tau)\,d\tau \nn\\
& \quad + \frac{e^{i\omega t}}{2i\omega} \left[  \wtil g_3(\omega, t)
 +
 \left(i \gamma \delta \nu - \gamma k^2  + i\nu k^2\right) \wtil g_0(\omega, t)
 + i\nu \wtil \beta(\omega, t) - \left(i \delta \nu - k^2\right) \wtil \alpha(\omega, t)\right] \\
& \quad -\frac{e^{-i\omega t}}{2i\omega}  \left[  \wtil g_3(-\omega, t)
 +
 \left(i \gamma \delta \nu - \gamma k^2 + i\nu k^2\right) \wtil g_0(-\omega, t)
 + i\nu \wtil \beta(-\omega, t) - \left(i \delta \nu -k^2\right) \wtil \alpha(-\omega, t)\right], \quad \text{Im}(\nu)\leq 0,
 \nn\\[2mm]
 \label{gr-nu}
\what u (-\nu,t)
&= \frac{e^{i\omega t}}{2i\omega} \left(\what{u}_1(-\nu)+ i\omega \what{u}_0(-\nu)\right) - \frac{e^{-i\omega t}}{2i\omega} \left( \what{u}_1(-\nu)-i\omega \what{u}_0(-\nu)\right) 
\nn\\
& \quad
+  \frac{e^{i\omega t} }{2i\omega}\int_0^t e^{-i\omega\tau} \what f(-\nu,\tau)\,d\tau -\frac{e^{-i\omega t}}{2i\omega} \int_0^t e^{i\omega\tau} \what f(-\nu,\tau)\,d\tau \nn\\
& \quad + \frac{e^{i\omega t}}{2i\omega} \left[  \wtil g_3(\omega, t)
-
 \left(i \gamma \delta \nu + \gamma k^2  + i\nu k^2\right) \wtil g_0(\omega, t)
 - i\nu \wtil \beta(\omega, t) + \left(i \delta \nu + k^2\right) \wtil \alpha(\omega, t)\right] \\
& \quad -\frac{e^{-i\omega t}}{2i\omega}  \left[  \wtil g_3(-\omega, t)
-
 \left(i \gamma \delta \nu + \gamma k^2 + i \nu k^2\right) \wtil g_0(-\omega, t)
- i\nu \wtil \beta(-\omega, t) + \left(i \delta \nu + k^2\right) \wtil \alpha(-\omega, t)\right], \quad \text{Im}(\nu)\geq 0.
\nn
}
Employing \eqref{gr-k}, \eqref{gr+nu} for the integral along $\p D_2$ and \eqref{gr-k}, \eqref{gr-nu} for the integral along $\p D_1$, we obtain two linear systems for $\wtil g_3, \wtil g_0$ which can be solved to yield
 the following solution formula for the linearized ``good'' Boussinesq problem \eqref{lgbous}: 
\aaa{\label{T-sol}
u(x,t)&= \frac{1}{2\pi} \int_{\R} \frac{e^{ikx}}{2i\omega}\left\{e^{i\omega t}\left[\what{u}_1(k)+ i\omega \what{u}_0(k)\right]-e^{-i\omega t}\left[ \what{u}_1(k)-i\omega \what{u}_0(k)\right]\right\}dk 
\nn\\ 
& \quad +\frac{1}{2\pi} \int_{\p D_2}  \frac{e^{ikx+i\omega t}}{2i\omega} \left\{\frac{(k-\nu)\left[\gamma(k+\nu)+i (\gamma\delta - k\nu )\right]}{\Delta_2(k)(k+\nu)} \left[i\omega \what{u}_0(-k)+\what{u}_1(-k)\right]\right\}\,dk
\nn\\
& \quad -\frac{1}{2\pi} \int_{\p D_2}  \frac{e^{ikx+i\omega t}}{2i\omega} \left\{ \frac{2ik(\nu^2+\gamma \delta)}{\Delta_2(k)(k+\nu)} \left[ i\omega \what{u}_0(\nu)+\what{u}_1(\nu)\right]\right\}\,dk 
\nn\\
& \quad + \frac{1}{2\pi} \int_{\p D_1}  \frac{e^{ikx-i\omega t}}{2i\omega} \left\{ \frac{(k+\nu)\left[\gamma(k-\nu)+i(\gamma\delta+k\nu)\right]}{\Delta_1(k)(k-\nu)} \left[i\omega \what{u}_0(-k)-\what{u}_1(-k)\right]\right\}\,dk
\nn\\
& \quad -\frac{1}{2\pi} \int_{\p D_1}  \frac{e^{ikx-i\omega t}}{2i\omega} \left\{\frac{2ik(\nu^2+\gamma \delta)}{\Delta_1(k)(k-\nu)} \left[ i\omega \what{u}_0(-\nu)-\what{u}_1(-\nu)\right]\right\}\,dk 
\nn\\
& \quad +  \frac{1}{2\pi} \int_{\R} \frac{e^{ikx}}{2i\omega} \left\{e^{i\omega t } \int_0^t e^{-i\omega\tau} \what f(k,\tau)\,d\tau - e^{-i\omega t} \int_0^t e^{i\omega\tau} \what f(k,\tau)\,d\tau  \right\}dk 
\nn\\
& \quad +\frac{1}{2\pi} \int_{\p D_2} \frac{e^{ikx+i\omega t}}{2i\omega} \left\{\frac{(k-\nu)\left[\gamma(k+\nu)+i (\gamma\delta - k\nu )\right]}{\Delta_2(k)(k+\nu)} \int_0^t e^{-i\omega\tau} \what{f}(-k,\tau)\,d\tau \right\}\,dk 
\nn\\
& \quad -\frac{1}{2\pi} \int_{\p D_2}  \frac{e^{ikx+i\omega t}}{2i\omega} \left\{ \frac{2ik(\nu^2+\gamma \delta)}{\Delta_2(k)(k+\nu)} \int_0^t e^{-i\omega\tau} \what{f}(\nu,\tau)\,d\tau \right\}\,dk 
\nn\\
& \quad - \frac{1}{2\pi} \int_{\p D_1}  \frac{e^{ikx-i\omega t}}{2i\omega} \left\{ \frac{(k+\nu)\left[\gamma(k-\nu)+i(\gamma\delta+k\nu)\right]}{\Delta_1(k)(k-\nu)} \int_0^t e^{i\omega\tau} \what{f}(-k,\tau)\,d\tau\right\}\,dk
\nn\\
& \quad +\frac{1}{2\pi} \int_{\p D_1}  \frac{e^{ikx-i\omega t}}{2i\omega} \left\{\frac{2ik(\nu^2+\gamma \delta)}{\Delta_1(k)(k-\nu)}  \int_0^t e^{i\omega\tau} \what{f}(-\nu,\tau)\,d\tau\right\}\,dk 
\nn\\
& \quad -\frac{1}{2\pi} \int_{\p D_2}  \frac{e^{ikx+i\omega t}}{i\omega}\cdot \frac{k(k-\nu)}{\Delta_2(k)} \left[(\nu+i\gamma) \wtil{\beta}(\omega,T) +i\nu(\nu+i\delta)\wtil{\alpha}(\omega,T)\right]\,dk
\nn\\
& \quad -\frac{1}{2\pi} \int_{\p D_1}  \frac{e^{ikx-i\omega t}}{i\omega} \cdot \frac{k(k+\nu)}{\Delta_1(k)} 
\left[
(\nu-i\gamma)\wtil{\beta}(-\omega,T) - i\nu(\nu-i\delta)\wtil{\alpha}(-\omega,T)
\right]\,dk
}
where we note that the radius $a$ associated with the contours $\p D_1, \p D_2$ in Figure \ref{fig:0} is chosen sufficiently large to ensure that the (finitely many) zeros of the quantities
\eee{
 \Delta_1(k) :=  -\gamma (k+\nu) + i (\gamma \delta -k\nu),
\quad
\Delta_2(k) := -\gamma (k-\nu) + i(\gamma \delta + k\nu)
\label{Del-2}
}
lie inside the circle of radius $a$ and center at $0$. Furthermore, thanks to exponential decay inside the regions enclosed by $\p D_1$ and $\p D_2$, we have been able to replace $t$ by the fixed value $T > t$ in the second argument of the transforms $\wtil a$ and $\wtil \beta$. Finally, analyticity and exponential decay inside those regions has allowed us to show that the integrals that arise from \eqref{gr-k}-\eqref{gr-nu} and involve the unknown quantities $\what u(-k, t), \what u(\pm \nu, t)$ vanish. 

For example, by Cauchy's theorem, we have
\eee{
\int_{\p D_1} 
e^{ikx} \, 
 \frac{k^2(i\nu+\gamma)+i\gamma\delta\nu}{\Delta_1(k)(k-\nu)} \, \what u(-k, t) \, dk 
 = \lim_{R\to\infty} \int_{C_R} e^{ikx} \, 
 \frac{k^2(i\nu+\gamma)+i\gamma\delta\nu}{\Delta_1(k)(k-\nu)} \, \what u(-k, t) \, dk =: \lim_{R\to\infty} J_R(x, t),
}
where $C_R$ is the quarter-circle of radius $R$ in the first quadrant. For $|k|\gg 1$, we have $\nu \simeq ik$, $\Delta_1 \simeq k^2$, thus 
\aaa{
|J_R(x, t)| 
&\lesssim \int_{C_R} e^{-\text{Im}(k) x}    |\what u(-k, t)|\,|dk| = \int_{C_R} e^{-\text{Im}(k) x}    \left|\int_0^\infty e^{iky} u(y, t) dy \right| 
|dk|
 \nn\\
 &\leq  |u(0, t)| 
	\int_{C_R} 
	\frac{e^{-\text{Im}(k) x}}{|k|}\, |dk| +
	\int_{C_R}
	\frac{e^{-\text{Im}(k) x}}{|k|}
	\int_0^\infty 
	e^{-\text{Im}(k) y} |u'(y, t)| dy\,|dk|  
}
after also integrating by parts. Then, parametrizing $C_R$ via $k=R e^{i\theta}, 0 \leq \theta \leq \frac{\pi}{2}$ and using Fubini's theorem, we deduce via the inequality $\sin \theta \geq \frac{2\theta}{\pi}$ for $0 \leq \theta \leq \frac{\pi}{2}$  that 
\aaa{
|J_R(x, t)| 
	&\lesssim
	|u(0, t)| 
	\int_0^{\frac{\pi}{2}}
	e^{-R x \sin \theta}\,d\theta + 
	\int_0^\infty |u'(y, t)| 
	\int_0^{\frac{\pi}{2}}
	e^{-R (x+y) \sin\theta }\,d\theta\,dy 
\nn\\
&\leq
|u(0, t)|  
\int_0^{\frac{\pi}{2}}
e^{- \frac{2Rx}{\pi}\,\theta }\,d\theta 
+
\int_0^\infty |u'(y, t)| 
\int_0^{\frac{\pi}{2}}
e^{- \frac{2R(x+y)}{\pi}\,\theta }\,d\theta \, dy
\nn\\
&=
|u(0, t)|  \, 
 \frac{\pi}{2Rx} \left[1-  e^{-Rx} \right] 
 +
 \int_0^\infty |u'(y, t)| \, \frac{\pi}{2R(x+y)} \left[1-  e^{-R(x+y)} \right] dy,
}
thus concluding that $\displaystyle \lim_{R\to\infty} J_R(x, t) = 0$ for each $x \in (0, \infty)$ and $t\in (0, T)$, as desired.

\bibliographystyle{myamsalpha}
\bibliography{references_DM}

\end{document}